\documentclass[11pt,reqno]{amsart}
\usepackage[utf8]{inputenc}
\usepackage[dvipsnames,svgnames,table]{xcolor}
\usepackage[colorlinks=true, pdfstartview=FitV, linkcolor=blue, citecolor=blue, urlcolor=blue]{hyperref}
\usepackage{booktabs}
\usepackage{amsmath,amsthm,amssymb}
\usepackage{enumitem}
\usepackage{mathtools}
\usepackage{cleveref}
\usepackage{comment}
\makeatletter
\renewcommand*\env@matrix[1][\arraystretch]{%
  \edef\arraystretch{#1}%
  \hskip -\arraycolsep
  \let\@ifnextchar\new@ifnextchar
  \array{*\c@MaxMatrixCols c}}
\makeatother
\usepackage{tikz}
\usetikzlibrary{calc}
\usetikzlibrary{matrix}
\usetikzlibrary{decorations.pathreplacing}
\newlength{\mathaxis}
\AtBeginDocument{\setlength{\mathaxis}{\fontdimen22\textfont2}}

\tikzset{
  cmplx/.style={
    x=1cm, y=8mm,
    baseline={([yshift=-\mathaxis]current bounding box.center)},
    every node/.append style={inner sep=2pt}
  }
}

\usetikzlibrary{arrows.meta, positioning}
\definecolor{Blue}{RGB}{24, 95, 165}
\definecolor{Red}{RGB}{216, 90, 48}
\definecolor{Green}{RGB}{61, 154, 80}
\tikzset{
 dot/.style = {
  circle, fill=black, inner sep=0pt,
  minimum size=5pt, outer sep=2pt
 },
 arr/.style = {-{Stealth[length=4pt, width=4pt]}, line width=0.6pt},
    barr/.style = {-{Stealth[length=4pt, width=4pt]}, line width=0.6pt,
  Blue, every node/.style={text=black}},
 rarr/.style = {-{Stealth[length=4pt, width=4pt]}, line width=0.6pt,
  Red, every node/.style={text=black}},
 garr/.style = {-{Stealth[length=4pt, width=4pt]}, line width=0.6pt,
  Green, every node/.style={text=black}},
 implies/.style = {
  ->, double, double distance=2pt,
  -{Implies}, shorten >=4pt, shorten <=4pt,
  line width=0.5pt, black
 },
    rdash/.style = {-{Stealth[length=4pt,width=4pt]}, line width=0.6pt, Red,
    dash pattern=on 3pt off 2pt},
    gdash/.style = {-{Stealth[length=4pt,width=4pt]}, line width=0.8pt, Green,
    dash pattern=on 6pt off 2pt},
   squig/.style = {
  -{Stealth[length=4.5pt, width=5pt]},
  line width=0.6pt,
  black!55,
  decoration={coil, aspect=0, segment length=9pt, amplitude=2pt},
  decorate
 }
}
\usepackage{soul}
\usepackage{bm}
\usepackage{graphicx}
\usepackage{mathrsfs}
\usepackage[margin=2.5cm]{geometry}
\usepackage[disable]{todonotes}
\usetikzlibrary{arrows,decorations.pathmorphing,shapes}
\newcounter{sarrow}

\newtheorem{theorem}{Theorem}[section]
\newtheorem{lemma}[theorem]{Lemma}
\newtheorem{proposition}[theorem]{Proposition}
\newtheorem{corollary}[theorem]{Corollary}

\theoremstyle{definition}
\newtheorem{definition}[theorem]{Definition}

\newtheorem{example}[theorem]{Example}
\newtheorem{question}[theorem]{Question}
\theoremstyle{remark}
\newtheorem{remark}[theorem]{Remark}
\numberwithin{equation}{section}

\newcommand{\F}{\mathcal F}

\newcommand{\N}{\mathbb N}
\newcommand{\Z}{\mathbb Z}

\newcommand{\mm}{\mathfrak m}
\newcommand{\Hom}{{\rm Hom}}
\newcommand{\Ext}{{\rm Ext}}

\DeclareMathOperator{\id}{id}

\newif\ifdraftnotes
\draftnotesfalse

\author[]{Daniel Bravo}
\address[D.~Bravo]
{Instituto de Ciencias Físicas y Matemáticas, Facultad de Ciencias, Universidad Austral de Chile, Valdivia, Chile}
\email{daniel.bravo@uach.cl}

\author[]{Federico Castillo}
\address[F.~Castillo]
{Matem\'aticas B\'asicas, Centro de Investigaci\'on en Matem\'aticas, Guanajuato, Mexico}
\email{federico.castillo@cimat.mx}

\author[]{Nelly Villamizar}
\address[N. Villamizar]
{Department of Mathematics, Swansea University, Swansea, United Kingdom}
\email{n.y.villamizar@swansea.ac.uk}

\date{\today}

\begin{document}

\title{
The Column-Diamond Resolution over Quadratic Monomial Rings
}

\begin{abstract}
We study infinite minimal free resolutions of modules $M = R/(x_{i_1},\ldots,x_{i_m})$ over quadratic monomial rings $R = k[x_1,\ldots,x_n]/I$.
We introduce the \emph{Column-Diamond construction}, a recursive diagrammatic procedure that builds an explicit free resolution of $M$ from columns and diamonds determined by the quadratic generators of $I$. We prove that this resolution is exact and minimal. As an application, the Betti numbers of $M$ are counted by valid words in a graph associated to $I$, which yields a combinatorial rational formula for the Poincar\'e series.
\end{abstract}
\maketitle

\section{Introduction}
Free resolutions of graded modules over polynomial rings and their quotients are a central theme of commutative algebra, with deep connections to algebraic geometry and combinatorics.
Over a polynomial ring, Hilbert's Syzygy Theorem bounds the length of a minimal free resolution by the number of variables, whereas over a quotient ring minimal free resolutions can be infinite and their structure remains far from understood; see the survey of McCullough and Peeva \cite{Survey-McculloughPeeva}.
Computing an initial segment of an infinite resolution often says little about the whole, and many techniques from the finite case break down, so explicit constructions are particularly valuable.

A particularly tractable class is that of \emph{quadratic monomial rings} $R=k[x_1,\ldots,x_n]/I$, where $I$ is generated by quadratic monomials.
Quadratic monomial ideals are encoded by graphs, and the systematic study of the algebras attached to graphs in this way is due in large part to Villarreal; see \cite[Chapter~6]{Villarreal2001}.
From this viewpoint, homological invariants of $R$ are read off from graph-theoretic data of $\mathsf G$.

Fröberg showed that these rings are Koszul \cite{Froberg}, which means that the base field has a linear minimal free resolution.
Furthermore, the same is true for every quotient $R/(Y)$ by a set $Y$ of variables \cite[Remark~1.11]{conca2014koszul}.
Backelin proved that the Poincaré series of $k$ over $R$ is rational \cite{backelin}, and Lescot extended this to all finitely generated $R$-modules \cite{lescot}.
Berglund \cite{berglund} described the denominator of the multigraded Poincaré series of a monomial ring through the homology of simplicial complexes associated to the generators of the ideal, and Berglund, Blasiak and Hersh \cite{berglund-2} reinterpreted this description in terms of lattices, obtaining simpler formulas in the quadratic case.
For edge ideals of graphs, Betti numbers and Hilbert series have been studied in graph-theoretic terms.
For example, the number of linear syzygies of the edge ideal of a simple graph $\mathsf G$ is $|E(L(\mathsf G))|-N_t$, where $L(\mathsf G)$ is the line graph and $N_t$ the number of triangles \cite{EliahouVillarreal}, \cite[Proposition~6.6.3]{Villarreal2001}.
Likewise, the Hilbert series of $S/I(\mathsf G)$ satisfies a deletion recursion on vertices \cite[Exercise~6.6.5]{Villarreal2001}.
Our results are in the same spirit, but for the infinite resolutions of the modules $R/(x_\Lambda)$ over $R$.

Bruns, Herzog and Vetter \cite{SyzygiesWalks} showed that $R/(Y)$ has a linear resolution for every set $Y$ of variables, but they construct it explicitly, with a basis indexed by ``neg-chains", only when $I$ is generated by the incomparable pairs of a poset and $Y$ corresponds to an order ideal.
Our Column-Diamond resolution is explicit for every quadratic monomial ideal and every $Y$, and our valid words reduce to their neg-chains in the poset case.

While many homological properties of quadratic monomial rings are thus well understood, explicit constructions of their resolutions are less common.
In this paper we work directly in the commutative setting and give a fully explicit construction of the minimal free resolution of $M=R/(x_{i_1},\ldots,x_{i_m})$, where each $x_{i_j}$ divides some generator of $I$.
Our construction is explicit enough that it allows us to give combinatorial formulas for the Poincaré series.

The central object of this paper is the \emph{Column–Diamond} construction, a recursive diagrammatic procedure that builds an infinite free complex from the combinatorial data of the quadratic monomial ideal $I$.
The construction proceeds by repeatedly attaching elementary building blocks, called \emph{columns} and \emph{diamonds}, according to simple local completion rules determined by the generators of $I$.
The resulting diagrams encode free modules and differentials in a natural way and, as we prove, give rise to free resolutions of these modules.

Our first main result (\Cref{thm:free-resolution}) shows that the complexes produced by the Column–Diamond construction are exact, and in fact minimal.
The proof proceeds by induction, using the linearity of the minimal free resolution of $M$ to reduce exactness to identifying the linear syzygies of each differential.
We show that these syzygies are generated by the column and diamond completions.

Our second main result (\Cref{thm:rational_generating_function}) describes the Column-Diamond resolution itself in terms of words.
We show that the vertices of the Column-Diamond diagram, that is, the basis elements of the resolution, are in bijection with certain valid words in a graph associated to the ideal $I$.
Thus the Column-Diamond resolution provides not only the Betti numbers, as counts of valid words, but an explicit basis of each free module indexed by these words, together with its differential.
As a consequence we recover the rationality of the Poincaré series, due to Backelin \cite{backelin} and Lescot \cite{lescot}, together with an explicit way to compute it.
Rationality can also be deduced from Hilbert series, using the Hilbert--Serre theorem and the recursion of \cite{Villarreal2001} (Remark~\ref{rem:hilbert}); the words give in addition an explicit basis and the transfer matrix.

Our approach works directly in the commutative quadratic monomial setting and focuses on the explicit construction of free resolutions.
The resulting Column-Diamond diagrams are of course closely related to the Koszul construction, but it has the flexibility to describe the quotient by any subset of the variables.

The paper is organized as follows.
Section \ref{sec:preliminaries} recalls basic notions concerning chain complexes, free resolutions, and quadratic monomial rings.
Section \ref{Chain-complex-procedure} introduces the Column-Diamond construction.
In Section \ref{sec:free-resolution} we prove that the chain complex produced by the Column-Diamond construction is exact, and hence gives a minimal free resolution of the corresponding module.
Section \ref{sec:betti} shows that the vertices of the Column-Diamond diagram are in bijection with certain valid words in a graph associated to $I$, and derives from this that the Poincaré series is rational. In Section~\ref{sec:injdim} we dualize the Column-Diamond resolution to exhibit explicit nonzero classes in $\Ext^i_R(R/(x_1),R)$ for infinitely many $i$ when $\mathsf G$ is a star.
Appendix~\ref{app:diamonds} contains the proof of a structural property of Column-Diamond diagrams and of the coherence of the signs of the diamonds. Appendix~\ref{app:schreyer} relates the Column-Diamond diagram to the Schreyer frames of La Scala and Stillman \cite{LaScala}, an approach that led us to the valid words.

\section{Preliminaries}\label{sec:preliminaries}

Throughout, $k$ is a field of arbitrary characteristic and $S=k[x_1,\dots,x_n]$ is standard graded, with $\deg x_i=1$.
We write $[n]=\{1,\dots,n\}$, and for a homogeneous ideal $I\subseteq S$ we let $\mm=(x_1,\dots,x_n)$ also denote the homogeneous maximal ideal of $R=S/I$.

\subsection{Free resolutions}

A chain complex $C$ of $R$-modules is a sequence of $R$-modules $C_i$, $i\in\Z$, together with homomorphisms $d_i\colon C_i\to C_{i-1}$, called \emph{differentials}, such that $d_{i-1}\circ d_i=0$ for all $i$.
Its homology modules are $H_i(C)=\ker(d_i)/\operatorname{im}(d_{i+1})$, and $C$ is \emph{exact} if $H_i(C)=0$ for all $i$; see \cite[Chapter~1]{Weibel}.
A \emph{free resolution} of an $R$-module $M$ is a complex $F\colon\cdots\xrightarrow{d_2}F_1\xrightarrow{d_1}F_0$ of free $R$-modules, with $F_i=0$ for $i<0$, together with a surjection $\varepsilon\colon F_0\to M$ such that the augmented complex $\cdots\to F_1\to F_0\xrightarrow{\varepsilon}M\to 0$ is exact.

When $M$ is finitely generated and graded, we may take every $F_i$ to be a graded free module of finite rank and every $d_i$ to be homogeneous of degree zero.
Such a resolution is \emph{minimal} if $d_i(F_i)\subseteq\mm F_{i-1}$ for all $i\geq 1$, that is, if no nonzero scalar appears in the matrices of the differentials.
Every finitely generated graded $R$-module has a minimal graded free resolution, which is unique up to isomorphism of complexes \cite[Theorem~3.5]{Survey-McculloughPeeva}.

The \emph{$i$-th Betti number} of $M$ is $\beta^R_i(M)=\operatorname{rank}F_i$, where $F$ is a minimal free resolution of $M$; equivalently, $\beta^R_i(M)=\dim_k\operatorname{Tor}^R_i(M,k)$ ; over $S$ see \cite[Corollary~2.5.7 and Definition~2.5.10]{Villarreal2001}.
We collect them in the \emph{Poincar\'e series}
\begin{equation}\label{eq:poincare}
	\mathsf{P}^R(M;z)=\sum_{i\geq 0}\beta^R_i(M)\,z^i\in\Z[[z]].
\end{equation}
When the ring is clear from context we omit it, and write $\beta_i(M)$ and $\mathsf{P}(M;z)$.
Over $S$ minimal free resolutions have length at most $n$ by Hilbert's Syzygy Theorem \cite[Theorem~2.5.6]{Villarreal2001}, but over a quotient $R$ they can be infinite.

A minimal graded free resolution $F$ of a module $M$ generated in degree $e$ is \emph{linear} if $F_i$ is generated in degree $e+i$ for every $i$, or equivalently, if every entry of the matrices of the differentials is a linear form.
The ring $R$ is \emph{Koszul} if the residue field $k=R/\mm$ has a linear resolution over $R$.

\subsection{Quadratic monomial rings}\label{sec:linearSyzygies}

A \emph{quadratic monomial ring} is a ring $R=S/I$ where $I$ is generated by monomials of degree two.
Such an ideal is encoded by a graph.
We let $\mathsf{G}=\mathsf{G}(I)$ be the graph on the vertex set $[n]$ with an edge $\{i,j\}$ whenever $x_ix_j\in I$, so that there is a loop at $i$ exactly when $x_i^2\in I$.
Conversely, every graph on $[n]$, possibly with loops, is $\mathsf{G}(I)$ for exactly one such ideal $I$, namely its edge ideal, where a loop at $i$ contributes the generator $x_i^2$.
For graphs without loops these are the edge ideals of \cite[Chapter~6]{Villarreal2001}, where graphs are simple by convention \cite[Section~6.1]{Villarreal2001}.
Allowing loops accounts for the generators $x_i^2$.

For $i\in[n]$ we write $N(i)$ for the set of neighbours of $i$ in $\mathsf{G}$, which contains $i$ exactly when there is a loop at $i$.
Thus $x_ix_j\in I$ if and only if $j\in N(i)$.

The following result is the main algebraic input of the paper.

\begin{theorem}\label{thm:koszul}
Let $R$ be a quadratic monomial ring.
Then $R$ is Koszul \cite{Froberg}, and for every set $Y$ of variables the module $R/(Y)$ has a linear minimal free resolution over $R$ \cite[Remark~1.11]{conca2014koszul}.
\end{theorem}

\begin{remark}\label{rem:linearSyzygies}
Theorem~\ref{thm:koszul} concerns resolutions over $R$.
It should not be confused with linearity of the resolution of $R$ itself over $S$, which for a graph without loops holds if and only if the complement of $\mathsf G$ is chordal \cite[Theorem~6.7.5]{Villarreal2001}, due to Fröberg \cite{Froberg1990}.
\end{remark}

We use Theorem~\ref{thm:koszul} in the following form.
Let $M=R/(Y)$ and let $F_i\xrightarrow{d_i}F_{i-1}\to\cdots\to F_0\xrightarrow{\varepsilon}M\to 0$ be a minimal partial free resolution, exact at $F_{i-1},\ldots,F_0$, with $F_j$ generated in degree $j$ for $j\leq i$.
Since it extends to a minimal free resolution of $M$, which is linear, $\ker(d_i)$ is generated by its elements of degree $i+1$, that is, by elements $\sum_p \ell_p e_p$ where the $e_p$ form a basis of $F_i$ and every $\ell_p$ is a linear form.
We call such elements \emph{linear syzygies}.

\medskip

The modules we resolve are the quotients of $R$ by ideals generated by variables.
For a nonempty set $\Lambda\subseteq[n]$ we write $x_\Lambda=\{x_\lambda:\lambda\in\Lambda\}$ and $M=R/(x_\Lambda)$, which is $R/(x_a)$ when $\Lambda=\{a\}$.
We always assume that $N(\lambda)\neq\emptyset$ for every $\lambda\in\Lambda$, that is, that every $x_\lambda$ divides some generator of $I$.
This is no loss of generality. Indeed, if $N(\lambda)=\emptyset$, then $R=R'[x_\lambda]$ for a quadratic monomial ring $R'$ in the remaining variables.
Since {$x_\lambda$ is a nonzero divisor on $R/(x_{\Lambda\setminus\{\lambda\}})$,  a minimal free resolution of $R'/(x_{\Lambda\setminus\{\lambda\}})$ over $R'$
can be extended to $R$ and tensored over $R$ with the two-term free
resolution of $R/(x_\lambda)$.
In particular, $\mathsf{P}(M;z)=(1+z)\,\mathsf{P}(R/(x_{\Lambda\setminus\{\lambda\}});z)$.

The following examples show that already for the two smallest graphs, a single loop and a single edge, the minimal free resolution of $M$ is infinite, and yet it can be written down explicitly.

\begin{example}\label{resolution-dual-number}
Let $R=k[x_1]/(x_1^2)$, whose graph is a single vertex with a loop, and let $M=R/(x_1)=k$.
Since $x_1^2=0$, the kernel of multiplication by $x_1$ on $R$ is its image $(x_1)$, so the complex
\[
\cdots\xrightarrow{\cdot x_1}R\xrightarrow{\cdot x_1}R\xrightarrow{\cdot x_1}R\xrightarrow{\cdot x_1}R\xrightarrow{\varepsilon}k\to 0
\]
is exact.
All its differentials are linear, so it is the minimal free resolution of $k$, and $\mathsf{P}(k;z)=1/(1-z)$.
\hfill$\diamond$
\end{example}

\begin{example}\label{resolution-polarized-dual-number}
Let $R=k[x_1,x_2]/(x_1x_2)$, whose graph is a single edge, and let $M=R/(x_1)$.
The kernel of multiplication by $x_1$ on $R$ is $(x_2)$ and the kernel of multiplication by $x_2$ is $(x_1)$, so the complex
\[
\cdots\xrightarrow{\cdot x_2}R\xrightarrow{\cdot x_1}R\xrightarrow{\cdot x_2}R\xrightarrow{\cdot x_1}R\xrightarrow{\varepsilon}M\to 0
\]
is exact.
As before, it is the minimal free resolution of $M$, and $\mathsf{P}(M;z)=1/(1-z)$.
\hfill$\diamond$
\end{example}

In the next section we describe the Column-Diamond construction, a combinatorial procedure that associates to every quadratic monomial ring $R$ and every $\Lambda$ as above a diagram, from which we read off the minimal free resolution of $R/(x_\Lambda)$.

\subsection{Rationality from linearity}\label{ssec:hilbert}

Before turning to the construction, we review how the linearity in Theorem~\ref{thm:koszul} already implies that $\mathsf{P}(M;z)$ is a rational function.

Let $N=\bigoplus_{j\geq 0}N_j$ be a finitely generated graded $R$-module generated in nonnegative degrees.
Each $N_j$ is a finite-dimensional vector space over $k$, and the \emph{Hilbert series} of $N$ is $\mathsf{H}(N;z)=\sum_{j\geq 0}\dim_k N_j\,z^j\in\Z[[z]]$.
For $i\geq 0$ we write $N(-i)$ for the module $N$ with its grading shifted by $i$, that is, $N(-i)_j=N_{j-i}$, so that $\mathsf{H}(N(-i);z)=z^i\,\mathsf{H}(N;z)$.

It is standard that Hilbert series alternate along exact sequences.
Let $F$ be the minimal free resolution of $M=R/(x_\Lambda)$ and write $\beta_i=\beta_i(M)$.
Since its differentials are homogeneous of degree zero, for every $j$ the augmented resolution restricts to an exact sequence of finite-dimensional vector spaces $\cdots\to(F_1)_j\to(F_0)_j\to M_j\to 0$.
Since $F_i$ is generated in degree $i$, we have $(F_i)_j=0$ for $i>j$, so this sequence is finite and the alternating sum of its dimensions vanishes.
Hence $\dim_k M_j=\sum_{i}(-1)^i\dim_k(F_i)_j$ for every $j$, which is the first equality below.
The key point is that $F_i$ is generated in degree $i$, so $F_i\cong R(-i)^{\beta_i}$, and therefore
\begin{align}
\mathsf{H}(M;z)
&=\sum_{i\geq 0}(-1)^i\,\mathsf{H}(F_i;z) \notag\\
&=\sum_{i\geq 0}(-1)^i\beta_i\,\mathsf{H}(R(-i);z) \notag\\
&=\sum_{i\geq 0}(-1)^i\beta_i\,z^i\,\mathsf{H}(R;z) \notag\\
&=\mathsf{P}(M;-z)\,\mathsf{H}(R;z). \label{eq:hilbert-poincare}
\end{align}
This is the only place where linearity is used: for an arbitrary graded resolution one has $F_i\cong\bigoplus_j R(-j)^{\beta_{ij}}$, and the same computation only recovers $\sum_{i,j}(-1)^i\beta_{ij}z^j$, in which the homological degree is lost.

Since $\mathsf{H}(R;z)$ has constant term $1$, it is invertible in $\Z[[z]]$, and \eqref{eq:hilbert-poincare} gives
\[\mathsf{P}(M;-z)=\mathsf{H}(M;z)/\mathsf{H}(R;z).\]
Both $R=S/I$ and $M\cong S/(I+(x_\Lambda))$ are quotients of $S$ by monomial ideals, so their Hilbert series count the monomials outside a monomial ideal, and such series are rational functions with denominator $(1-z)^n$.
Hence $\mathsf{P}(M;-z)$, and therefore $\mathsf{P}(M;z)$, is a rational function.

\begin{example}\label{ex:hilbert-fibonacci}
Let $R=k[x_1,x_2,x_3]/(x_1x_2,x_1x_3)$ and $M=R/(x_1)$, which will be our running example in Section~\ref{Chain-complex-procedure}.
A monomial is nonzero in $R$ exactly when it is divisible by neither $x_1x_2$ nor $x_1x_3$, that is, when it is either a monomial in $x_2,x_3$ or a power $x_1^a$ with $a\geq 1$.
Hence $\mathsf{H}(R;z)=\frac{1}{(1-z)^2}+\frac{z}{1-z}=\frac{1+z-z^2}{(1-z)^2}$.
Since $M\cong k[x_2,x_3]$, we have $\mathsf{H}(M;z)=1/(1-z)^2$, and \eqref{eq:hilbert-poincare} gives $\mathsf{P}(M;-z)=1/(1+z-z^2)$, that is, $\mathsf{P}(M;z)=1/(1-z-z^2)$.
Its coefficients $1,1,2,3,5,8,\dots$ are the Fibonacci numbers, in agreement with Examples~\ref{ex:free-res-from-diag} and~\ref{ex:path_3}.
\hfill$\diamond$
\end{example}

\section{Column-Diamond Construction} \label{Chain-complex-procedure}%

Let $R=S/I$ be a quadratic monomial ring, and let $\Lambda\subseteq[n]$ be nonempty with $N(\lambda)\neq\emptyset$ for every $\lambda\in\Lambda$, as in Section~\ref{sec:preliminaries}.
We recursively associate a diagram to the pair $(R,M)$, where $M=R/(x_\Lambda)$.
Although our primary focus is on modules of the form $R/(x_a)$, the construction applies to every such $\Lambda$.
This recursive procedure will be called the \emph{Column-Diamond construction}.

\subsection{Graphical conventions%
} \label{conventions}

The following conventions describe how %
free modules and homomorphisms will be represented diagrammatically.
They are intended to simplify both the description of the procedure and the final construction of the corresponding diagram.

While the conventions are reversible (allowing conversion back to algebraic form), we will work in diagrammatic form throughout.
Our conventions are as follows:
\begin{enumerate}[topsep=1ex,itemsep=1ex]
\item Each copy of the ring $R$ is represented by the symbol $\bullet$ (henceforth referred to as a \emph{vertex}), and conversely.%

\item A direct sum of copies of $R$ is represented by a horizontal juxtaposition of the symbol $\bullet$, and conversely.

\item %
Each map $R \xrightarrow{\cdot x_i} R$, given by multiplication by $x_i$, is represented by the diagram
\begin{tikzpicture}[baseline=(L.center)]
\node[dot] (L) at (0,   0) {};
\node[dot] (R) at (0.9, 0) {};
\draw[arr] (L) -- node[above, font=\footnotesize]{$i$} (R);
\end{tikzpicture}\ .
{An arrow labelled $-i$ represents multiplication by $-x_i$.}
\item %
Maps between free modules are represented by collections of labelled arrows, interpreted as matrices with respect to the chosen bases.
For example, the diagrams below represent maps $R^2\to R$ and $R\to R^
2$,  respectively.

\[\begin{tikzpicture}[baseline=(current bounding box.center)]
\node[dot] (b0) at (0,1) {};
\node[dot] (b1) at (1,1) {};
\node[dot] (c0) at (0.5,0) {};
\draw[arr] (b0) -- node[left,font=\footnotesize]{$i$} (c0);
\draw[arr] (b1) -- node[right,font=\footnotesize]{$j$} (c0);
\end{tikzpicture}
\qquad
\begin{tikzpicture}[baseline=(current bounding box.center)]
\node[dot] (c0) at (0.5,1) {};
\node[dot] (b0) at (0,0) {};
\node[dot] (b1) at (1,0) {};
\draw[arr] (c0) -- node[left,font=\footnotesize]{$i$\,} (b0);
\draw[arr] (c0) -- node[right,font=\footnotesize]{\;$j$} (b1);
\end{tikzpicture}
\]

\smallskip

\noindent When convenient, we replace labels by colors: we may assign a distinct color to each index $i$ and represent multiplication by $x_i$ using an arrow of the corresponding color.

\end{enumerate}

These diagrams also encode the homological grading. %
If a free module occurs in degree $m$, then the vertices corresponding to its basis elements are said to be in degree $m$.

We now describe the recursive procedure that produces the Column-Diamond diagrams.

\subsection{Construction procedure.} \label{ssec:diagram-procedure}
Let $R$, $\Lambda=\{\lambda_1,\dots,\lambda_r\}$ and $M=R/(x_\Lambda)$ be as above.
The Column-Diamond construction associates a diagram to the pair $(R,M)$ through the following four steps.

\begin{enumerate}[label={\bfseries Step \arabic*},topsep=1ex,itemsep=1ex]
\item\hspace{-0.15cm}{\textbf :} \textbf{diagram collection.} \label{Set-D}
Construct a set $D$ of elementary diagrams as follows.%

\begin{enumerate}[topsep=1ex,itemsep=1ex]
\item For each quadratic monomial $x_ix_j \in I$ with $i\neq j$, include these two ``column'' diagrams in $D$:
\begin{equation}\tag{D1}\label{i-j-diferente}
\begin{tikzpicture}[baseline=(current bounding box.center)]
\begin{scope}[shift={(-1.8, 0)}]
\node[dot] (L1) at (0, 1.4) {};
\node[dot] (L2) at (0, 0.7) {};
\node[dot] (L3) at (0, 0) {};
\draw[arr] (L1) -- node[left, font=\footnotesize]{$j$} (L2);
\draw[arr] (L2) -- node[left, font=\footnotesize]{$i$} (L3);
\end{scope}
\node at (0, 0.7) {\text{and}};
\begin{scope}[shift={(1.8, 0)}]
\node[dot] (R1) at (0, 1.4) {};
\node[dot] (R2) at (0, 0.7) {};
\node[dot] (R3) at (0, 0) {};
\draw[arr] (R1) -- node[left, font=\footnotesize]{$i$} (R2);
\draw[arr] (R2) -- node[left, font=\footnotesize]{$j$} (R3);
\node[font=\footnotesize] at (0.5, 0.7) {.};
\end{scope}
\end{tikzpicture}
\end{equation}
\item If $x_i^2 \in I$, include this ``repeated column'' diagram in $D$:
\begin{equation}\tag{D2}\label{i-j-igual}
\begin{tikzpicture}[baseline=(current bounding box.center)]
 \node[dot] (T) at (0, 1.4) {};
 \node[dot] (M) at (0, 0.7) {};
 \node[dot] (B) at (0, 0) {};
 \draw[arr] (T) -- node[left, font=\footnotesize]{$i$} (M);
 \draw[arr] (M) -- node[left, font=\footnotesize]{$i$} (B);
 \node[font=\scriptsize] at (0.5, 0.7) {.};
\end{tikzpicture}
\end{equation}
\item For each pair $i< j$ such that $x_ix_j \not \in I$, include this ``diamond'' diagram (or its symmetric version, i.e.
exchanging $i$ with $j$) in $D$:

\begin{equation}
\tag{D3}
\label{i-j-no-esta}
\begin{tikzpicture}[baseline=(current bounding box.center)]

\node[dot] (T) at (0, 1.4) {};
\node[dot] (BL) at (-0.5, 0.7) {};
\node[dot] (BR) at ( 0.5, 0.7) {};
\node[dot] (B) at (0, 0) {};

\draw[arr] (T) -- node[left, font=\footnotesize]{$-j$} (BL);
\draw[arr] (T) -- node[right, font=\footnotesize]{\, $i$} (BR);
\draw[arr] (BL) -- node[left, font=\footnotesize]{$i$} (B);
\draw[arr] (BR) -- node[right, font=\footnotesize]{$j$} (B);

\node[font=\footnotesize] at (0.25, -0.15) {.};

\end{tikzpicture}
\end{equation}
We adopt the convention that the upper arrow with the
larger index carries the negative sign.
More generally, the signs may be chosen differently,
provided that the two paths of the diamond carry
opposite signs.

\item No diagram is associated to any $x_i^2 \not \in I$.
\end{enumerate}

\item\hspace{-0.15cm}{\textbf :} \textbf{initial map.}\label{initialMap} The construction begins with the map
\[
d_1 \colon \bigoplus_{j=1}^{r} R \xrightarrow{\begin{bmatrix} x_{\lambda_1} & \cdots & x_{\lambda_r} \end{bmatrix}} R\ .
\]
This map represents a fan of arrows
converging to a single vertex, and corresponds to a degree-lowering
map from degree $1$ to degree $0$, as follows:
\begin{equation}\label{fan-like-map}
\begin{tikzpicture}[baseline=(current bounding box.center)]

\node[dot] (a1) at (-1.8,1.0) {};
\node[dot] (a2) at (-0.6,1.0) {};
\node at (0.6,1.0) {$\cdots$};
\node[dot] (ar) at (1.8,1.0) {};
\node[dot] (b0) at (0,0) {};

\draw[arr] (a1) -- node[left, pos=0.5,
    font=\footnotesize]{$\lambda_1$} (b0);

\draw[arr] (a2) -- node[right, pos=0.35,
    font=\footnotesize]{$\lambda_2$} (b0);

\draw[arr] (ar) -- node[right, pos=0.5,
    font=\footnotesize]{\hspace{0.3cm}$\lambda_{r}$} (b0);

\end{tikzpicture}
\end{equation}

\item\hspace{-0.15cm}{\textbf :}\label{completion-rules} \textbf{completion rules.}
Given $d\geq 0$ and a map from degree $d+1$ to degree $d$,   apply the following  ``diagram completion'' rules to produce a map from degree $d+2$ to degree $d+1$.

\begin{enumerate}[topsep=1ex,itemsep=0.5ex]
\item \textbf{Column completion:}\label{columnCompletion}
For each arrow
\begin{tikzpicture}[baseline=(current bounding box.center)]
\node[dot] (T) at (0, 0.7) {};
\node[dot] (B) at (0, 0) {};
\draw[arr] (T) -- node[left, font=\footnotesize]{$i$} (B);
\end{tikzpicture} from degree $d+1$ to degree $d$,
extend the diagram to degree $d+2$ by attaching  all the column diagrams in $D$ having this arrow as their lower arrow.

\item \textbf{Diamond completion:}\label{diamondCompletion} For every pair of arrows
\begin{tikzpicture}[baseline=(current bounding box.center)]
\node[dot] (BL) at (-0.5, 0.7) {};
\node[dot] (BR) at ( 0.5, 0.7) {};
\node[dot] (B) at (0, 0) {};
\draw[arr] (BL) -- node[left, font=\footnotesize]{$i$} (B);
\draw[arr] (BR) -- node[right, font=\footnotesize]{$j$} (B);
\end{tikzpicture}  with $i\neq j$ and $x_ix_j\notin I$, extend the diagram
to degree $d+2$ by attaching the corresponding diamond
from $D$, and thus forming the diagram \eqref{i-j-no-esta}.

The signs of the new arrows are chosen so that the two paths of the diamond carry opposite signs (see Remark \ref{rem:signed-diamonds}).

\item\label{no-rep} \textbf{No repetitions.} %

If two or more new arrows with the same label have
the same target, their source vertices are identified
and the arrows are merged into a single arrow.
All other arrows from the identified source vertices
are retained.
This rule applies to both column and diamond completions.

\end{enumerate}
\item\hspace{-0.15cm}{\textbf :}\label{iteration} \textbf{iteration.}
Repeat the completion rules using the newly constructed
map from degree $d+2$ to degree $d+1$.
Continue recursively to obtain the maps in higher degrees.
\end{enumerate}
This completes the Column-Diamond construction.

\begin{example} \label{ejemplo-2-cubo}
Let $R=k[x_1,x_2,x_3]/(x_1x_2,x_1x_3)$, and $M=R/(x_1)$.
We use red, green and blue arrows to represent
multiplication by $x_2$, $x_3$ and $x_1$,
respectively. A dashed arrow indicates multiplication
by the negative of the corresponding variable.
That is
\[
\begin{tikzpicture}[baseline=(L.center)]
\node[dot] (L) at (0, 0) {};
\node[dot] (R) at (0.9, 0) {};
\draw[garr] (L) -- node[above,black, font=\footnotesize]{$-3$} (R);
\end{tikzpicture}
\quad\text{or}\quad
\begin{tikzpicture}[baseline=(L.center)]
\node[dot] (L) at (0, 0) {};
\node[dot] (R) at (0.9, 0) {};
\draw[garr,dashed] (L) -- (R);
\end{tikzpicture}
\]
represents the map from $R$ to $R$ given by multiplication by $-x_3$.

Applying \ref{Set-D} of the construction, since $x_1x_2, x_1x_3\in I$ and $x_2x_3\notin I$, we get the following collection diagrams:
\[
D = \left\{
\begin{tikzpicture}[baseline=(current bounding box.center)]
\node[dot] (T) at (0, 1.4) {};
\node[dot] (M) at (0, 0.7) {};
\node[dot] (B) at (0, 0) {};
\draw[barr] (T) -- node[left, font=\footnotesize]{$1$} (M);
\draw[rarr] (M) -- node[left, font=\footnotesize]{$2$} (B);
\end{tikzpicture}\;, \qquad
\begin{tikzpicture}[baseline=(current bounding box.center)]
\node[dot] (T) at (0, 1.4) {};
\node[dot] (M) at (0, 0.7) {};
\node[dot] (B) at (0, 0) {};
\draw[rarr] (T) -- node[left, font=\footnotesize]{$2$} (M);
\draw[barr] (M) -- node[left, font=\footnotesize]{$1$} (B);
\end{tikzpicture}\;, \qquad
\begin{tikzpicture}[baseline=(current bounding box.center)]
\node[dot] (T) at (0, 1.4) {};
\node[dot] (M) at (0, 0.7) {};
\node[dot] (B) at (0, 0) {};
\draw[barr] (T) -- node[left, font=\footnotesize]{$1$} (M);
\draw[garr] (M) -- node[left, font=\footnotesize]{$3$} (B);
\end{tikzpicture}\;, \qquad
\begin{tikzpicture}[baseline=(current bounding box.center)]
\node[dot] (T) at (0, 1.4) {};
\node[dot] (M) at (0, 0.7) {};
\node[dot] (B) at (0, 0) {};
\draw[garr] (T) -- node[left, font=\footnotesize]{$3$} (M);
\draw[barr] (M) -- node[left, font=\footnotesize]{$1$} (B);
\end{tikzpicture}\;, \qquad
\begin{tikzpicture}[baseline=(current bounding box.center)]
\node[dot] (T) at (0, 1.4) {};
\node[dot] (BL) at (-0.5, 0.7) {};
\node[dot] (BR) at ( 0.5, 0.7) {};
\node[dot] (B) at (0, 0) {};
\draw[garr] (T) -- node[left, font=\footnotesize]{$-3$} (BL);
\draw[rarr] (T) -- node[right, font=\footnotesize]{$2$} (BR);
\draw[rarr] (BL) -- node[left, font=\footnotesize]{$2$} (B);
\draw[garr] (BR) -- node[right, font=\footnotesize]{$3$} (B);
\end{tikzpicture}
\right\}\ .
\]

By \ref{initialMap}, the initial map is
multiplication by $x_1$, represented by

\begin{equation}\label{eq:initialdiagramEx1}
\begin{tikzpicture}[baseline=(current bounding box.center)]
\node[dot] (T) at (0, 0.7) {};
\node[dot] (B) at (0, 0) {};
\node[right, font=\small] at (0.3, 0.7) {degree $d=1$};
\node[right, font=\small] at (0.3, 0) {degree $d=0$\ .};
\draw[barr] (T) -- node[left, font=\small]{$1$} (B);
\end{tikzpicture}
\end{equation}
Applying the completion rules to the initial diagram \eqref{eq:initialdiagramEx1}, we obtain two column completions from $D$: %

\begin{equation}\label{diagram-Ex1-step3}
\begin{tikzpicture}[baseline=(current bounding box.center)]

\node (B1) at (-2, 0){
\begin{tikzpicture}[baseline=(current bounding box.center)]
\node[dot] (L1) at (0, 1.4) {};
\node[dot] (L2) at (0, 0.7) {};
\node[dot] (L3) at (0, 0) {};
\draw[rarr] (L1) -- node[left, font=\footnotesize]{$2$} (L2);
\draw[barr] (L2) -- node[left, font=\footnotesize]{$1$} (L3);
\end{tikzpicture}
};

\node (B2) at (2, 0){
\begin{tikzpicture}[baseline=(current bounding box.center)]
\node[dot] (R1) at (0, 1.4) {};
\node[dot] (R2) at (0, 0.7) {};
\node[dot] (R3) at (0, 0) {};
\draw[garr] (R1) -- node[left, font=\small]{$3$} (R2);
\draw[barr] (R2) -- node[left, font=\small]{$1$} (R3);
\end{tikzpicture}
};

\node (A) at (0, -2){
\begin{tikzpicture}[baseline=(current bounding box.center)]
\node[dot] (S1) at (0, 0.7) {};
\node[dot] (S2) at (0, 0) {};
\draw[barr] (S1) -- node[left, font=\small]{$1$} (S2);
\end{tikzpicture}
};

\draw[squig, color=black] (B1) -- (-0.05,-1.6);
\draw[squig, color=black] (B2) -- (A);

\node (E) at (3, -1) {$=$};

\node (R) at (5.5, -1){
\begin{tikzpicture}[baseline=(current bounding box.center)]
\node[dot] (P1) at ( 0, 0) {};
\node[dot] (P2) at ( 0, 0.8) {};
\node[dot] (P3) at (-0.8, 1.6) {};
\node[dot] (P4) at ( 0.8, 1.6) {};
\draw[barr] (P2) -- (P1);
\draw[rarr] (P3) -- (P2);
\draw[garr] (P4) -- (P2);
\node[right, font=\small] at (1.1, 0) {$d=0$\ .};
\node[right, font=\small] at (1.1, 0.8) {$d=1$};
\node[right, font=\small] at (1.1, 1.6) {$d=2$};
\end{tikzpicture}
};

\end{tikzpicture}
\end{equation}

Now, applying \ref{iteration}, we extend the diagram \eqref{diagram-Ex1-step3}.

There are three diagrams from $D$ to add: two columns and one diamond.
For clarity, the labels on the arrows are omitted, and we rely only on the color coding.

The extension yields a diagram from degree $3$ to degree $2$:

\[
\begin{tikzpicture}[baseline=(current bounding box.center)]

\node (B1) at (-2, 0){
\begin{tikzpicture}[baseline=(current bounding box.center)]
\node[dot] (L1) at (0, 1.4) {};
\node[dot] (L2) at (0, 0.7) {};
\node[dot] (L3) at (0, 0) {};
\draw[barr] (L1) -- (L2);
\draw[rarr] (L2) -- (L3);
\end{tikzpicture}
};

\node (B2) at (0, 0){
\begin{tikzpicture}[baseline=(current bounding box.center)]
\node[dot] (T) at (0, 1.4) {};
\node[dot] (BL) at (-0.5, 0.7) {};
\node[dot] (BR) at ( 0.5, 0.7) {};
\node[dot] (B) at (0, 0) {};
\draw[garr,dashed] (T) -- (BL);
\draw[rarr] (T) -- (BR);
\draw[rarr] (BL) -- (B);
\draw[garr] (BR) -- (B);
\end{tikzpicture}
};

\node (B3) at (2, 0){
\begin{tikzpicture}[baseline=(current bounding box.center)]
\node[dot] (R1) at (0, 1.4) {};
\node[dot] (R2) at (0, 0.7) {};
\node[dot] (R3) at (0, 0) {};
\draw[barr] (R1) -- (R2);
\draw[garr] (R2) -- (R3);
\end{tikzpicture}
};

\node (A) at (0, -2.6){
\begin{tikzpicture}[baseline=(current bounding box.center)]
\node[dot] (P1) at ( 0, 0) {};
\node[dot] (P2) at ( 0, 0.8) {};
\node[dot] (P3) at (-0.7, 1.6) {};
\node[dot] (P4) at ( 0.7, 1.6) {};
\draw[barr] (P2) -- (P1);
\draw[rarr] (P3) -- (P2);
\draw[garr] (P4) -- (P2);
\end{tikzpicture}
};

\draw[squig, color=black] (B1) -- (-0.78,-1.6);
\draw[squig, color=black] (B2) -- (0,-1.7);
\draw[squig, color=black] (B3) -- (0.78,-1.6);

\node at (3, -1) {$=$};

\node at (5.5, -1){
\begin{tikzpicture}[baseline=(current bounding box.center)]
\node[dot] (P1) at ( 0, 0) {};
\node[dot] (P2) at ( 0, 0.8) {};
\node[dot] (P3) at (-0.7, 1.6) {};
\node[dot] (P4) at ( 0.7, 1.6) {};
\node[dot] (P5) at (-1.5, 2.5) {};
\node[dot] (P6) at ( 0, 2.5) {};
\node[dot] (P7) at ( 1.5, 2.5) {};
\draw[barr] (P2) -- (P1);
\draw[rarr] (P3) -- (P2);
\draw[garr] (P4) -- (P2);
\draw[barr] (P5) -- (P3);
\draw[garr,dashed] (P6) -- (P3);
\draw[rarr] (P6) -- (P4);
\draw[barr] (P7) -- (P4);
\node[right, font=\small] at (2.0, 1.5) {$d=2$};
\node[right, font=\small] at (2.0, 2.5) {$d=3$};
\end{tikzpicture}
};

\end{tikzpicture}
\]

We illustrate one more iteration, extending the diagram to degree $4$.
Six diagrams contribute to the completion, but the  ``No repetitions'' rule \ref{no-rep} identifies two of the new vertices. Hence there are five vertices in degree $4$.
\[
\begin{tikzpicture}[baseline=(current bounding box.center)]

\node (B1) at (-2.5, 0){
\begin{tikzpicture}[baseline=(current bounding box.center)]
\node[dot] (A51) at (0, 1.4) {};
\node[dot] (A52) at (0, 0.7) {};
\node[dot] (A53) at (0, 0) {};
\draw[rarr] (A51) -- (A52);
\draw[barr] (A52) -- (A53);
\end{tikzpicture}
};

\node (B2) at (-1.5, 0){
\begin{tikzpicture}[baseline=(current bounding box.center)]
\node[dot] (A51) at (0, 1.4) {};
\node[dot] (A52) at (0, 0.7) {};
\node[dot] (A53) at (0, 0) {};
\draw[garr] (A51) -- (A52);
\draw[barr] (A52) -- (A53);
\end{tikzpicture}
};

\node (B3) at (-0.5, 0){
\begin{tikzpicture}[baseline=(current bounding box.center)]
\node[dot] (A51) at (0, 1.4) {};
\node[dot] (A52) at (0, 0.7) {};
\node[dot] (A53) at (0, 0) {};
\draw[barr] (A51) -- (A52);
\draw[garr] (A52) -- (A53);
\end{tikzpicture}
};

\node (B4) at (0.5, 0){
\begin{tikzpicture}[baseline=(current bounding box.center)]
\node[dot] (A51) at (0, 1.4) {};
\node[dot] (A52) at (0, 0.7) {};
\node[dot] (A53) at (0, 0) {};
\draw[barr] (A51) -- (A52);
\draw[rarr] (A52) -- (A53);
\end{tikzpicture}
};

\node (B5) at (1.5, 0){
\begin{tikzpicture}[baseline=(current bounding box.center)]
\node[dot] (A51) at (0, 1.4) {};
\node[dot] (A52) at (0, 0.7) {};
\node[dot] (A53) at (0, 0) {};
\draw[rarr] (A51) -- (A52);
\draw[barr] (A52) -- (A53);
\end{tikzpicture}
};

\node (B6) at (2.5, 0){
\begin{tikzpicture}[baseline=(current bounding box.center)]
\node[dot] (A61) at (0, 1.4) {};
\node[dot] (A62) at (0, 0.7) {};
\node[dot] (A63) at (0, 0) {};
\draw[garr] (A61) -- (A62);
\draw[barr] (A62) -- (A63);
\end{tikzpicture}
};

\node (A) at (0, -3.2){
\begin{tikzpicture}[baseline=(current bounding box.center)]
\node[dot] (P1) at ( 0, 0) {};
\node[dot] (P2) at ( 0, 0.8) {};
\node[dot] (P3) at (-0.7, 1.6) {};
\node[dot] (P4) at ( 0.7, 1.6) {};
\node[dot] (P5) at (-1.5, 2.5) {};
\node[dot] (P6) at ( 0, 2.5) {};
\node[dot] (P7) at ( 1.5, 2.5) {};
\draw[barr] (P2) -- (P1);
\draw[rarr] (P3) -- (P2);
\draw[garr] (P4) -- (P2);
\draw[barr] (P5) -- (P3);
\draw[garr,dashed] (P6) -- (P3);
\draw[rarr] (P6) -- (P4);
\draw[barr] (P7) -- (P4);
\end{tikzpicture}
};

\draw[squig, color=black] (-2.3, -1) -- (-1.6, -1.8);
\draw[squig, color=black] (-1.5,-1) -- (-1.5, -1.7);
\draw[squig, color=black] (-0.5,-0.9) -- (-0.1, -1.8);
\draw[squig, color=black] (0.4,-0.9) -- (0.1, -1.8);
\draw[squig, color=black] (1.5,-1) -- ( 1.5, -1.7);
\draw[squig, color=black] (2.3,-1) -- ( 1.6, -1.8);

\node at (3.5, -1) {$=$};

\node at (7, -1){
\begin{tikzpicture}[baseline=(current bounding box.center)]
\node[dot] (P1) at ( 0, 0) {};
\node[dot] (P2) at ( 0, 0.8) {};
\node[dot] (P3) at (-0.7, 1.6) {};
\node[dot] (P4) at ( 0.7, 1.6) {};
\node[dot] (P5) at (-1.5, 2.5) {};
\node[dot] (P6) at ( 0, 2.5) {};
\node[dot] (P7) at ( 1.5, 2.5) {};
\node[dot] (P8) at (-2.2, 3.4) {};
\node[dot] (P9) at (-0.7, 3.4) {};
\node[dot] (P10) at ( 0, 3.4) {};
\node[dot] (P11) at ( 0.7, 3.4) {};
\node[dot] (P12) at ( 2.2, 3.4) {};
\draw[barr] (P2) -- (P1);
\draw[rarr] (P3) -- (P2);
\draw[garr] (P4) -- (P2);
\draw[barr] (P5) -- (P3);
\draw[garr,dashed] (P6) -- (P3);
\draw[rarr] (P6) -- (P4);
\draw[barr] (P7) -- (P4);
\draw[rarr] (P8) -- (P5);
\draw[garr] (P9) -- (P5);
\draw[barr] (P10) -- (P6);
\draw[rarr] (P11) -- (P7);
\draw[garr] (P12) -- (P7);
\node[right, font=\small] at (2.5, 2.5) {$d=3$};
\node[right, font=\small] at (2.5, 3.4) {$d=4$};
\end{tikzpicture}
};

\end{tikzpicture}
\]
Continuing with the procedure (that is, applying \ref{iteration}), we obtain the following diagram:

\begin{equation}\label{diagrama-2-cubo}
\begin{tikzpicture}[baseline=(current bounding box.center),xscale=0.5, yscale=0.65]
\node[dot] (n10) at (12,1) {};
\node[dot] (n20) at (12,2) {};
\node[dot] (n31) at (10,3) {};
\node[dot] (n32) at (14,3) {};
\node[dot] (n41) at (4,4) {};
\node[dot] (n42) at (12,4) {};
\node[dot] (n43) at (20,4) {};
\node[dot] (n51) at (2.5,5) {};
\node[dot] (n52) at (5.5,5) {};
\node[dot] (n53) at (12,5) {};
\node[dot] (n54) at (18.5,5) {};
\node[dot] (n55) at (21.5,5) {};
\node[dot] (n61) at (1,6) {};
\node[dot] (n62) at (4,6) {};
\node[dot] (n63) at (7,6) {};
\node[dot] (n64) at (11,6) {};
\node[dot] (n65) at (13,6) {};
\node[dot] (n66) at (17,6) {};
\node[dot] (n67) at (20,6) {};
\node[dot] (n68) at (23,6) {};
\node[dot] (n71) at (0,7) {};
\node[dot] (n72) at (2,7) {};
\node[dot] (n73) at (4,7) {};
\node[dot] (n74) at (6,7) {};
\node[dot] (n75) at (8,7) {};
\node[dot] (n76) at (10,7) {};
\node[dot] (n77) at (12,7) {};
\node[dot] (n78) at (14,7) {};
\node[dot] (n79) at (16,7) {};
\node[dot] (n710) at (18,7) {};
\node[dot] (n711) at (20,7) {};
\node[dot] (n712) at (22,7) {};
\node[dot] (n713) at (24,7) {};
\node at (12,8) {$\vdots$};

\draw[barr] (n20) -- (n10);
\draw[rarr] (n31) -- (n20);
\draw[garr] (n32) -- (n20);
\draw[barr] (n41) -- (n31);
\draw[garr,dashed] (n42) -- (n31);
\draw[rarr] (n42) -- (n32);
\draw[barr] (n43) -- (n32);
\draw[rarr] (n51) -- (n41);
\draw[garr] (n52) -- (n41);
\draw[barr] (n53) -- (n42);
\draw[rarr] (n54) -- (n43);
\draw[garr] (n55) -- (n43);
\draw[barr] (n61) -- (n51);
\draw[garr,dashed] (n62) -- (n51);
\draw[rarr] (n62) -- (n52);
\draw[barr] (n63) -- (n52);
\draw[rarr] (n64) -- (n53);
\draw[garr] (n65) -- (n53);
\draw[barr] (n66) -- (n54);
\draw[garr,dashed] (n67) -- (n54);
\draw[rarr] (n67) -- (n55);
\draw[barr] (n68) -- (n55);
\draw[rarr] (n71) -- (n61);
\draw[garr] (n72) -- (n61);
\draw[barr] (n73) -- (n62);
\draw[rarr] (n74) -- (n63);
\draw[garr] (n75) -- (n63);
\draw[barr] (n76) -- (n64);
\draw[garr,dashed] (n77) -- (n64);
\draw[rarr] (n77) -- (n65);
\draw[barr] (n78) -- (n65);
\draw[rarr] (n79) -- (n66);
\draw[garr] (n710) -- (n66);
\draw[barr] (n711) -- (n67);
\draw[rarr] (n712) -- (n68);
\draw[garr] (n713) -- (n68);
\end{tikzpicture}
\end{equation}

\end{example}

\medskip

We call the graded diagram produced by the Column-Diamond construction of $(R,M)$ the \emph{Column-Diamond diagram} of $M$.

\smallskip

Before associating free modules and differentials
to these diagrams, we  record a structural property.
The construction builds diamonds from the bottom up,
while the following lemma shows how they can also
be recognised from above or along a path of two arrows.

\begin{lemma}\label{lem:diamond-views}
In the Column-Diamond diagram the following hold.
\begin{enumerate}
\item\label{lem:diam-upperarrows} Any two distinct arrows out of the same vertex carry distinct labels and are the upper arrows of a diamond.
\item\label{lem:diam-bnotc} Every path $P\xrightarrow{c}Q\xrightarrow{b}R$ with {$b\neq c$ and} $x_bx_c\notin I$ lies in a diamond.
\item\label{lem:diam-squares} {There is no path $P\xrightarrow{b}Q\xrightarrow{b}R$ with $x_b^2\notin I$. %
}
\end{enumerate}
\end{lemma}

The proof, given in Appendix \ref{app:diamonds}, is an induction on the degree in which both statements are proved together, since each of them uses the other in lower degrees.

\begin{corollary}\label{cor:unique-arrow}
In the Column-Diamond diagram there is at most one arrow between two given vertices.
\end{corollary}

\begin{proof}
Suppose that there are two distinct arrows from a vertex $p$ to a vertex $q$.
By Lemma~\ref{lem:diamond-views}\eqref{lem:diam-upperarrows} they carry distinct labels $a\neq b$ and are the upper arrows of a diamond, whose lower arrows $q\xrightarrow{b}s$ and $q\xrightarrow{a}s$ are again two distinct arrows from one vertex into one vertex, one degree below $q$.
Repeating the argument we reach a vertex of degree $1$ with two distinct arrows, which contradicts the fact that each vertex of degree $1$ has a single arrow, the one of the initial map \eqref{fan-like-map}.
\end{proof}

Note that a single elementary diagram of $D$ has at most two upper arrows, whereas a vertex produced by the ``No repetitions'' rule may have more, and two of its arrows may lie in no common elementary diagram, being joined only through a third one.

\begin{remark}\label{rem:signed-diamonds}
The signs of a diamond are chosen so that its two paths carry opposite signs.
Explicitly, if the lower arrows carry $\varepsilon x_i$ and $\delta x_j$, the upper arrows labeled $j$ and $i$, which land on the vertices where $\varepsilon x_i$ and $\delta x_j$ start, carry signs $\varepsilon'$ and $\delta'$ with $\varepsilon'\varepsilon=-\delta'\delta$.
When both lower arrows are positive, the upper arrow with the larger label is the negative one, which is the convention of \eqref{i-j-no-esta}.
A vertex produced by the ``No repetitions'' rule may be the top of several diamonds, since by Lemma \ref{lem:diamond-views} any two of its arrows are the upper arrows of a diamond.
We then give the positive sign to its arrow with the smallest label, and the relations of its diamonds determine the remaining signs.
That these relations are compatible is also shown in Appendix \ref{app:diamonds}.
\end{remark}

In the next section we explain how to associate free modules and differentials to such diagrams, yielding sequences of free modules that will later be shown to be free resolutions.

\section{From the Column-Diamond Diagram to a Free Resolution
}
\label{sec:free-resolution}

The Column-Diamond construction of \Cref{Chain-complex-procedure} produces an infinite graded diagram associated to the pair $(R,M)$, where $R$ is a quadratic monomial ring and $M=R/(x_\Lambda)$.
We now explain how to extract from this diagram a sequence of free $R$-modules and homomorphisms.
The main result of this section shows that this sequence is the minimal free resolution of $M$.

\subsection{Free modules and differentials}

For each \(i\ge 0\), let \(\mathcal V_i\) denote the set of vertices in degree \(i\) of the Column-Diamond diagram, and set

\begin{equation}\label{free-module}
F_i=\bigoplus_{p\in \mathcal V_i} R\,e_p,
\end{equation}
where $\{e_p \colon p\in\mathcal V_i\}$ is a basis indexed by the vertices of degree $i$.

The differential $d_i\colon F_i\to F_{i-1}$ is determined vertex-by-vertex using the outgoing arrows of the diagram.

By construction, every vertex $p\in\mathcal V_i$ with $i\ge 1$ has one or more outgoing arrows, each pointing to a vertex in $\mathcal V_{i-1}$.

\begin{definition}\label{def:differential}
For each vertex $p \in \mathcal{V}_i$ ($i \geq 1$), define $
d_i\colon F_i\to F_{i-1}
$
on each basis element $e_p$ as follows:
if the outgoing arrows from $p$ go to $q_1,\ldots,q_r \in \mathcal{V}_{i-1}$, with labels $j_1,\ldots,j_r$ and signs $\delta_1, \ldots,\delta_r$, respectively, then
 \begin{equation} \label{eq:differential}
   d_i(e_p) = \sum_{t=1}^{r} \delta_t x_{j_t}\, e_{q_t},
 \end{equation}
where each $\delta_t \in\{\pm1\}$ is the sign carried by the corresponding arrows in the diagram.

\end{definition}

Let $\varepsilon\colon F_0=R\to M$ be the canonical projection.
Then, by construction, $\ker\varepsilon=\operatorname{im}(d_1)=(x_\Lambda)$.

 \begin{example}\label{ex:free-res-from-diag}
In the Column-Diamond diagram of $(R,M)$ constructed in Example~\ref{ejemplo-2-cubo}, where
$
R=k[x_1,x_2,x_3]/(x_1x_2,x_1x_3)
$
and
$
M=R/(x_1),
$
the first few vertex sets have cardinalities
$
1,\;1,\;2,\;3,\;5
$
vertices, respectively.
Hence
$$
F_0=R,\qquad
F_1=R,\qquad
F_2=R^2,\qquad
F_3=R^3,\qquad
F_4=R^5.
$$
Reading off the arrows from the diagram yields
\begin{align*}
d_1&=
\begin{bmatrix}
x_1
\end{bmatrix},&&
d_2=
\begin{bmatrix}
x_2 & x_3
\end{bmatrix},
\\
d_3& =
\begin{bmatrix}
x_1 & -x_3 & 0 \\
0 & x_2 & x_1
\end{bmatrix}, &&
d_4=
\begin{bmatrix}
x_2 & x_3 & 0 & 0 & 0 \\
0 & 0 & x_1 & 0 & 0 \\
0 & 0 & 0 & x_2 & x_3
\end{bmatrix}.
\end{align*}
Thus the first few free modules have ranks
$$
b_0=1,\quad b_1=1,\quad b_2=2,\quad b_3=3,\quad b_4=5,
$$
the beginning of the Fibonacci sequence.
This is no coincidence; see Section~\ref{sec:betti}.
\hfill$\diamond$
\end{example}

\begin{lemma}\label{lem:chain-complex}
 The maps obtained from the Column-Diamond diagram satisfy
 $
 d_i\circ d_{i+1}=0
 $
 for all $i\geq 1$.
\end{lemma}

\begin{proof}
Let $p'\in\mathcal V_{i+1}$ and $q\in\mathcal V_{i-1}$.
The coefficient of $e_q$ in $d_i(d_{i+1}(e_{p'}))$ is a sum over the paths $p'\xrightarrow{c}p\xrightarrow{k}q$, where each path contributes $x_cx_k$ times the product of the signs of its two arrows.
We distinguish three cases according to the labels of the path.

\noindent\textbf{Case 1:} $x_cx_k\in I$.
Then the contribution of the path vanishes in $R$.

\noindent\textbf{Case 2:} $c\neq k$ and $x_cx_k\notin I$.
By Lemma~\ref{lem:diamond-views} the path lies in a diamond
\[
\begin{tikzpicture}[baseline=(current bounding box.center)]
 \node[dot] (T) at (0, 1.4) {};
 \node[right] at (0, 1.6) {\ $p'$};
 \node[dot] (BL) at (-0.5, 0.7) {};
 \node[left] at (-0.6, 0.7) {$p$\ };
 \node[dot] (BR) at ( 0.5, 0.7) {};
 \node[right] at (0.6, 0.7) {\ $p_2$};
 \node[dot] (B) at (0, 0) {};
 \node[right] at (0, -0.15) {\ $q$};
 \draw[arr] (T) -- node[left, font=\footnotesize]{$c$\;} (BL);
 \draw[arr] (T) -- node[right, font=\footnotesize]{$\; k$} (BR);
 \draw[arr] (BL) -- node[left, font=\footnotesize]{$k$} (B);
 \draw[arr] (BR) -- node[right, font=\footnotesize]{$c$} (B);
\end{tikzpicture}
\]
and since the arrows out of a vertex carry distinct labels by Lemma~\ref{lem:diamond-views}\eqref{lem:diam-upperarrows}, $p'\xrightarrow{k}p_2\xrightarrow{c}q$ is the only other path from $p'$ to $q$ with labels $c$ and $k$.
By Remark~\ref{rem:signed-diamonds} the two paths carry opposite signs, so their contributions add up to $x_cx_ke_q-x_kx_ce_q=0$, by commutativity.

\noindent\textbf{Case 3:} $c=k$ and $x_c^2\notin I$.
This case does not occur by Lemma~\ref{lem:diamond-views}\eqref{lem:diam-squares}. %
Hence the paths of Case 2 cancel in pairs and the remaining paths contribute zero, so $d_i(d_{i+1}(e_{p'}))=0$.
\end{proof}

{We give each $F_i$ the grading in which the basis elements
$e_p$, for $p\in\mathcal V_i$, have degree $i$.
Thus $F_i\cong R(-i)^{|\mathcal V_i|}$ as graded $R$-modules.}

\begin{theorem}\label{thm:free-resolution}
The sequence obtained from the Column-Diamond diagram of $(R,M)$, with $M=R/(x_\Lambda)$,
\[
\cdots \xrightarrow{d_3} F_2 \xrightarrow{d_2} F_1 \xrightarrow{d_1} F_0 \xrightarrow{\varepsilon} M \to 0\ ,
\]
is a free resolution of $M$ over $R$, where $\varepsilon\colon F_0=R\to M$ is the canonical projection.
\end{theorem}

\begin{proof}
		It suffices to show that $\ker(d_i)\subseteq\operatorname{im}(d_{i+1})$ for every $i\ge 1$.
		This together with \Cref{lem:chain-complex} implies that the sequence is exact.
		First, notice that by construction, $\ker \varepsilon = \operatorname{im}(d_1) = (x_\Lambda)$, so the complex is exact at $F_0$.

		We proceed by induction on $i\geq 1$, assuming that the augmented complex is exact at $F_{i-1},\ldots,F_0$ and $M$.

		By our grading, $F_j$ is generated in degree $j$ and every $d_j$ is homogeneous of degree zero.
		Moreover, every entry of $d_j$ is a variable up to sign, so $d_j(F_j)\subseteq\mm F_{j-1}$.
		Hence $F_i\to\cdots\to F_0\to M\to 0$ is the beginning of a minimal free resolution of $M$, which by Theorem~\ref{thm:koszul}, in the form stated after Remark~\ref{rem:linearSyzygies}, is linear.
		Thus $\ker(d_i)$ is generated by its linear syzygies, and it suffices to show that each of them lies in $\operatorname{im}(d_{i+1})$.

		For an arrow $p'\to p$ of the diagram we write $\sigma(p'\to p)$ for its sign.
		This is well defined by Corollary~\ref{cor:unique-arrow}.
		By the ``No repetitions'' rule, for every vertex $p$ and every label $j$ there is at most one arrow labelled $j$ into $p$, when it exists we denote its source by $\pi(p,j)\in\mathcal V_{i+1}$

		Let $v=\sum_{p\in\mathcal V_i}\bigl(\sum_j\alpha_p^j x_j\bigr)e_p$, with $\alpha_p^j\in k$, be a linear syzygy in $\ker(d_i)$.
		Since $v$ has (internal) degree $i+1$ and $F_{i+1}$ is generated in degree $i+1$, we must find coefficients $c_{p'}\in k$, and vertices $p'\in\mathcal V_{i+1}$, with $v=d_{i+1}\bigl(\sum_{p'}c_{p'}e_{p'}\bigr)$.

		Fix $q\in\mathcal V_{i-1}$ and let $P(q)$ be the set of $p\in\mathcal V_i$ with an arrow into $q$, labelled $\ell_p$.
		The coefficient of $e_q$ in $d_i(v)$ is the class in $R$ of the polynomial
		\begin{equation}\label{eq:coeff-q}
			f_q=\sum_{p\in P(q)}\sum_j\sigma(p\to q)\,\alpha_p^j\,x_jx_{\ell_p}\in S,
		\end{equation}
		and $d_i(v)=0$ means that $f_q\in I$ for every $q$.
		Since $I$ is generated by monomials, $f_q\in I$ if and only if the coefficient in $f_q$ of every monomial $x_jx_\ell\notin I$ is zero.

		We first show that if $\alpha_p^j\neq 0$,
		then $\pi(p,j)$ exists.

		Let $p\xrightarrow{\ell}q$ be any arrow out of $p$, which exists since $i\geq 1$.
		If $x_jx_\ell\in I$, the column diagram of $D$ with lower label $\ell$ and upper label $j$ (one of \eqref{i-j-diferente} if $j\neq \ell$, \eqref{i-j-igual} if $j=\ell$) is attached to $p\xrightarrow{\ell}q$ by the column completion \eqref{columnCompletion}, producing an arrow labelled $j$ into $p$.
		Suppose instead that $x_jx_\ell\notin I$, so that the coefficient of $x_jx_\ell$ in $f_q$ vanishes.
		A term of \eqref{eq:coeff-q} contributes to this monomial exactly when $\{j',\ell_{p'}\}=\{j,\ell\}$ as multisets.
		If $j=\ell$, such a term comes from an arrow labelled $\ell$ into $q$, and $p\xrightarrow{\ell}q$ is the only one, so the coefficient is $\sigma(p\to q)\alpha_p^\ell\neq0$, a contradiction.
		If $j\neq \ell$, the terms come from $p\xrightarrow{\ell}q$ with $j'=j$, and from the arrow labelled $j$ into $q$, if it exists, with $j'=\ell$.
		If that arrow did not exist the coefficient would be $\sigma(p\to q)\alpha_p^j\neq0$. So, there is an arrow $p_2\xrightarrow{j}q$, and the diamond completion \eqref{diamondCompletion} applied to $p\xrightarrow{\ell}q$ and $p_2\xrightarrow{j}q$ produces an arrow labelled $j$ into $p$.
		In all cases $\pi(p,j)$ exists.

		We now define the coefficients $c_{p'}$.

		Let $p'\in\mathcal V_{i+1}$.
		The vertex $p'$ may be the top of several diamonds, since its outgoing arrows can arise from different completions.

		Let $p'\xrightarrow{a}p_1$ and $p'\xrightarrow{b}p_2$ be distinct arrows out of $p'$.
		By Lemma~\ref{lem:diamond-views}\eqref{lem:diam-upperarrows}, $a\neq b$, $x_ax_b\notin I$, and they are the upper arrows of a diamond with bottom vertex $q$, so that $p_1\xrightarrow{b}q$ and $p_2\xrightarrow{a}q$.
		These are the only arrows into $q$ labelled $b$ and $a$, so the coefficient of $x_ax_b$ in $f_q$ is
		\[
		\sigma(p_1\to q)\,\alpha_{p_1}^a+\sigma(p_2\to q)\,\alpha_{p_2}^b ,
		\]
		which vanishes since $x_ax_b\notin I$.
		By Remark~\ref{rem:signed-diamonds} and Lemma~\ref{lem:signs-compatible}, the two paths of this diamond carry opposite signs,
		\[
		\sigma(p'\to p_1)\,\sigma(p_1\to q)=-\,\sigma(p'\to p_2)\,\sigma(p_2\to q),
		\]
		and multiplying the vanishing coefficient above by $\sigma(p'\to p_1)\sigma(p_1\to q)$ gives
		\begin{equation}\label{eq:c-well-defined}
			\sigma(p'\to p_1)\,\alpha_{p_1}^a=\sigma(p'\to p_2)\,\alpha_{p_2}^b .
		\end{equation}
		Hence
		\[
		c_{p'}=\sigma(p'\to p)\,\alpha_p^j,\qquad\text{for any arrow } p'\xrightarrow{j}p,
		\]
		is well defined, since by \eqref{eq:c-well-defined} any two arrows out of $p'$ give the same value.

Set $w=\sum_{p'\in\mathcal V_{i+1}}c_{p'}e_{p'}$.
		We claim that $d_{i+1}(w)=v$.

		By Definition~\ref{def:differential},
		\[
		d_{i+1}(w)=\sum_{p'\in\mathcal V_{i+1}}\ \sum_{p'\xrightarrow{j}p}c_{p'}\,\sigma(p'\to p)\,x_j\,e_p
		=\sum_{p'\xrightarrow{j}p}\alpha_p^j\,x_j\,e_p ,
		\]
		the last sum running over all arrows from degree $i+1$ to degree $i$.
		Each pair $(p,j)$ for which $\pi(p,j)$ exists contributes exactly $\alpha_p^jx_je_p$, since there is one arrow labelled $j$ into $p$, and by the first step above, every pair for which $\pi(p,j)$ does not exist has $\alpha_p^j=0$.
		Therefore $d_{i+1}(w)=\sum_p\sum_j\alpha_p^jx_je_p=v$.

		Thus every linear syzygy lies in $\operatorname{im}(d_{i+1})$, so $\ker(d_i)\subseteq\operatorname{im}(d_{i+1})$, which completes the induction.
	\end{proof}

 \begin{definition} \label{Column-Diamond-resolution}
The resolution obtained from the Column-Diamond diagram is called the \textit{Column-Diamond resolution}
\end{definition}
\begin{remark}\label{rem:minimal}
By the definition of the differential in Equation \eqref{eq:differential}, every nonzero entry of $d_i$ is, up to sign, one of the variables, so $\operatorname{im}(d_i)\subseteq (x_1,\ldots,x_n)F_{i-1}$ for
every $i>0$.
Hence the Column-Diamond resolution is minimal, and in particular $\operatorname{rank}F_i$ is the $i$-th Betti number of $M$.
\end{remark}

\section{Combinatorial formula for Poincaré series}
\label{sec:betti}
Recall from Section~\ref{sec:preliminaries} that the $i$-th Betti number $\beta_i(M)$ of a finitely generated graded $R$-module $M$ is the rank of $F_i$ in a minimal free resolution $\cdots \to F_1 \to F_0 \to M$, and that it does not depend on the choice of minimal free resolution since $\beta_i(M) = \dim_k \operatorname{Tor}_i^R(M,k)$.
We are interested in the Poincaré series $\mathsf{P}(M;z) = \sum_{i \geq 0} \beta_i(M) z^i$.
Its rationality is due to Backelin \cite{backelin} and Lescot \cite{lescot}, and the Column-Diamond resolution gives a new proof of it, which also provides an explicit way to compute it.

\begin{theorem}
\label{thm:rational_generating_function}
Let $R$ be a quadratic monomial ring and let $\Lambda\subseteq[n]$ be nonempty.
Then the Poincaré series $\mathsf{P}(R/(x_\Lambda);z)$ is a rational function.
\end{theorem}

\begin{remark}\label{rem:hilbert}
Rationality also follows from Hilbert series.
View $R$ and $M=R/(x_\Lambda)$ as finitely generated graded $S$-modules, whose Hilbert series $H_R(t)$ and $H_M(t)$ are rational by the Hilbert--Serre theorem \cite[Theorem~4.1.3]{Villarreal2001}.
The Column-Diamond resolution is linear, so $F_i\cong R(-i)^{\beta_i(M)}$, and additivity of Hilbert series on this exact complex gives $H_M(t)=H_R(t)\,\mathsf P(M;-t)$ in $\Z[[t]]$.
Since $H_R(-z)$ has constant term $1$, we get $\mathsf P(M;z)=H_M(-z)/H_R(-z)$.
Note that $M$ is the quotient ring of the graph $\mathsf G\setminus\Lambda$.
If $\mathsf G$ has no loops, both series can be computed with the recursion $H_{\mathsf G}(t)=H_{\mathsf G\setminus x}(t)+\frac{t}{1-t}H_{\mathsf G\setminus(N(x)\cup\{x\})}(t)$ of \cite[Exercise~6.6.5]{Villarreal2001}, with $H_\emptyset=1$.
If $x$ has a loop, the factor $\frac{t}{1-t}$ is replaced by $t$.
\end{remark}

The interest of rationality comes from a basic fact in enumerative combinatorics \cite[Section 4.1]{stanely2011enumerative}.
If $\sum_{k \geq 0} a_k z^k = A(z)/Q(z)$ with $Q(z) = 1 - C_1 z - \cdots - C_d z^d$, then $a_k = C_1 a_{k-1} + \cdots + C_d a_{k-d}$ for all $k \geq \max(d, \deg A + 1)$.
In Example~\ref{ex:path_3} below we make the recursion and the initial conditions explicit.

To prove Theorem~\ref{thm:rational_generating_function} we use the combinatorial description of the resolution.
We will count each vertex of the diagram by assigning a special path from the root, which we encode as a word.
This reduces the computation of $\mathsf{P}(M;z)$ to the combinatorial problem of counting certain valid words.

\subsection{Valid words}
Let $\mathsf{G} = (V, E)$ be a finite graph, possibly with loops.
We assume that $V = \{1, 2, \dots, n\}$ with its natural order, which plays a role below.
For $v \in V$ let $N(v)$ be the set of neighbours of $v$, which contains $v$ when there is a loop at $v$.
We consider words in the alphabet $V$, written by juxtaposition as $W = w_1 w_2 \cdots w_k$, and we denote the empty word by $\emptyset$.
Fix a nonempty set $\Lambda \subseteq V$; in our application it is the set of indices of the variables generating the ideal {$(x_\Lambda)$} defining $M$.
To every word $W$ we attach a set $S_\Lambda(W) \subseteq V$, defined recursively by
\begin{equation*}
S_\Lambda(\emptyset) = \Lambda \qquad \text{and} \qquad  S_\Lambda(Wv) = N(v) \cup \{u \in S_\Lambda(W) : u < v\}.
\end{equation*}

\begin{definition}
\label{def:valid_word}
A word $W = w_1 \cdots w_k$ is \emph{$(\mathsf{G},\Lambda)$-valid} if $w_j \in S_\Lambda(w_1 \cdots w_{j-1})$ for $j = 1, \dots, k$.
The set of all $(\mathsf{G},\Lambda)$-valid words is denoted $\mathscr{V}_{\mathsf{G},\Lambda}$, and we write $\mathscr{V}_{\mathsf{G},a}$ when $\Lambda = \{a\}$.
When $\mathsf{G}$ and $\Lambda$ are clear from context we simply say that $W$ is valid.
\end{definition}

\begin{example}\label{ex:valid-words}
Consider the graph
\[
\mathsf{G}=\;
\begin{tikzpicture}[baseline=(current bounding box.center)]
\node[dot, label={[font=\scriptsize]below:$3$}] (v3) at (0,0) {};
\node[dot, label={[font=\scriptsize]below:$1$}] (v1) at (0.8,0) {};
\node[dot, label={[font=\scriptsize]below:$2$}] (v2) at (1.6,0) {};
\node[dot, label={[font=\scriptsize]below:$4$}] (v4) at (2.4,0) {};
\draw[line width=0.6pt] (v3) -- (v1) -- (v2) -- (v4);
\end{tikzpicture}
\]
and let $\Lambda=\{1,4\}$, so that
\[
N(1)=\{2,3\},\quad N(2)=\{1,4\},\quad N(3)=\{1\},\quad N(4)=\{2\}.
\]
We check that $1324$ is valid by computing the sets of its prefixes.
We have $S_\Lambda(\emptyset)=\{1,4\}$, which contains $1$, and
\begin{align*}
S_\Lambda(1)&=N(1)\cup\{u\in\{1,4\} : u<1\}=\{2,3\},\\
S_\Lambda(13)&=N(3)\cup\{u\in\{2,3\} : u<3\}=\{1\}\cup\{2\}=\{1,2\},\\
S_\Lambda(132)&=N(2)\cup\{u\in\{1,2\} : u<2\}=\{1,4\}\cup\{1\}=\{1,4\}.
\end{align*}
Each of these sets contains the next letter of $1324$, so the word is valid.
Note that the letter $2$ belongs to $S_\Lambda(13)$ although it is not a neighbour of $3$.
It is validated by the $1$ in the first position, and the only letter between them is $3>2$.

The word $1234$, with the same letters, is not valid.
Indeed, $S_\Lambda(12)=N(2)\cup\{u\in\{2,3\} : u<2\}=\{1,4\}$ does not contain $3$, since the $1$ that could validate $3$ is followed by the smaller letter $2$.

Finally, consider the word $41$.
Since $S_\Lambda(\emptyset)=\Lambda$, a valid word must start with an element of $\Lambda$, and indeed $4\in\Lambda$.
Then
\[
S_\Lambda(4)=N(4)\cup\{u\in\Lambda : u<4\}=\{2\}\cup\{1\}=\{1,2\},
\]
which contains $1$.
So $41$ is valid although $1$ is not a neighbour of $4$. The letter $1$ belongs to $\Lambda$ and
remains in $S_\Lambda(4)$ because $1<4$. \hfill$\diamond$
\end{example}

The recursion defining $S_\Lambda(W)$ can be unwound as follows.
Each letter of $W$ contributes its neighbours, and $\Lambda$ contributes its elements, but a contributed element survives only if it is smaller than all the later letters.

\begin{lemma}\label{lem:closed-form-S}
Let $W=w_1\cdots w_k$ be a word and $u\in V$.
Then $u\in S_\Lambda(W)$ if and only if either $u\in N(w_j)$ for some $j$ such that $w_{j+1},\dots,w_k$ are all larger than $u$, or $u\in\Lambda$ and $w_1,\dots,w_k$ are all larger than $u$.
\end{lemma}

\begin{proof}
We argue by induction on $k$.
For $k=0$ both sides say that $u\in\Lambda$.
For the induction step, $S_\Lambda(Wv)=N(v)\cup\{u\in S_\Lambda(W)\colon u<v\}$.
The condition of the statement for $Wv$ holds either with $j=k+1$, which means $u\in N(v)$, or with $j\leq k$ or with $\Lambda$.
In the latter cases it says exactly that the condition holds for $W$ and that $u<v$, which by induction means $u\in S_\Lambda(W)$ and $u<v$.
\end{proof}

\begin{example}\label{ex:closed-form-S}
Let $\mathsf{G}$ be the cycle on $\{1,\dots,5\}$ with edges $\{i,i+1\}$, indices modulo $5$, and let $\Lambda=\{1\}$.
To compute $S_\Lambda(12345)$ we read the letters from right to left.
The letter $5$ contributes $N(5)=\{1,4\}$, the letter $4$ contributes the elements of $N(4)=\{3,5\}$ smaller than $5$, that is $3$, the letter $3$ contributes $2$, and the letter $2$ contributes $1$.
The letter $1$ and the set $\Lambda$ contribute nothing, since no element of $N(1)=\{2,5\}$ is smaller than $2$ and no element of $\Lambda$ is smaller than $1$.
Hence $S_\Lambda(12345)=\{1,2,3,4\}$.
Similarly, $S_\Lambda(15)=\{1,4\}\cup\{2\}=\{1,2,4\}$.
\hfill$\diamond$
\end{example}

By Remark~\ref{rem:minimal}, $\beta_k(M)$ is the number of vertices of degree $k$ of the Column-Diamond diagram of $M$, so to compute $\mathsf{P}(M;z)$ we count these vertices.
We do so by attaching a word to each vertex.

\begin{definition}\label{def:word-of-vertex}
Let $p$ be a vertex of degree $k$ of the Column-Diamond diagram.
Following the arrows, $p$ is joined to the vertex of degree $0$ by paths $p=p_k\to p_{k-1}\to\cdots\to p_0$.
If the arrow $p_j\to p_{j-1}$ is labeled $w_j$, we call $w_1w_2\cdots w_k$ a \emph{word of $p$}, so that words are read from degree $0$ upwards.
The \emph{word} $\omega(p)$ of $p$ is the largest of its words in the lexicographic order.
\end{definition}

By the ``No repetitions'' rule a vertex has at most one incoming arrow with a given label.
So a word $w_1\cdots w_k$ determines at most one path, namely the one that starts with the arrow labeled $w_1$ into the vertex of degree $0$ and continues upwards along the arrows labeled $w_2,\dots,w_k$.
In particular, distinct paths from $p$ have distinct words, and distinct vertices have no word in common.

A vertex may have several words, since the Column-Diamond diagram is not a tree.
For instance, the vertex of degree $3$ at the top of the diamond in Example~\ref{ejemplo-2-cubo} has the words $123$ and $132$, so its word is $132$.
In Lemma~\ref{lem:vertex-words} we show that the words $\omega(p)$ are exactly the valid words.

\begin{lemma}\label{lem:exchange}
Let $p$ be a vertex of the Column-Diamond diagram.
\begin{enumerate}
\item\label{lem:1words} If a word of $p$ has two consecutive letters $a\neq b$ with $x_ax_b\notin I$, then exchanging them gives another word of $p$.
\item\label{lem:2words} Any two words of $p$ are obtained from each other by a sequence of exchanges as in~\eqref{lem:1words}.
\end{enumerate}
\end{lemma}

\begin{proof}
For \eqref{lem:1words}, suppose that $a$ and $b$ are the letters in positions $j$ and $j+1$, so that the path contains the arrows $p_{j+1}\xrightarrow{b}p_j\xrightarrow{a}p_{j-1}$.
By Lemma~\ref{lem:diamond-views}\eqref{lem:diam-bnotc} these two arrows lie in a diamond, whose other side is a path $p_{j+1}\xrightarrow{a}p'\xrightarrow{b}p_{j-1}$.
Replacing one side by the other gives a path from $p$ whose word is obtained by exchanging $a$ and $b$.

For \eqref{lem:2words} we argue by induction on the degree of $p$, the case of degree $0$ being trivial.
Let $Wu$ and $W'u'$ be two words of $p$, whose last letters correspond to arrows $p\xrightarrow{u}q$ and $p\xrightarrow{u'}q'$.
If these arrows coincide, then $W$ and $W'$ are words of $q$ and we conclude by induction.
Otherwise, by Lemma~\ref{lem:diamond-views}\eqref{lem:diam-upperarrows}, they are the upper arrows of a diamond, so $u\neq u'$, $x_ux_{u'}\notin I$, and there is a vertex $s$ with arrows $q\xrightarrow{u'}s$ and $q'\xrightarrow{u}s$.
If $Y$ is a word of $s$, then $Yu'$ is a word of $q$ and $Yu$ is a word of $q'$, so by induction $W$ can be turned into $Yu'$ and $W'$ into $Yu$ by exchanges.
Hence $Wu$ can be turned into $Yu'u$, then into $Yuu'$ by exchanging the last two letters, and then into $W'u'$.
\end{proof}

We can now prove the main connection between valid words and vertices of the diagram.

\begin{lemma}\label{lem:vertex-words}
For every $k\geq 0$, the map $p\mapsto\omega(p)$ is a bijection between the vertices of degree $k$ of the Column-Diamond diagram and the $(\mathsf{G},\Lambda)$-valid words of length $k$.
\end{lemma}

\begin{proof}
The map is injective, since distinct vertices have no word in common.
It remains to show that $\omega(p)$ is valid for every vertex $p$, and that every valid word is of the form $\omega(p)$.

Let $\omega(p)=w_1\cdots w_k$, fix a position $r$ and let $u=w_r$.
We show that $u\in S_\Lambda(w_1\cdots w_{r-1})$.
Let $m<r$ be the last position with $w_m\in N(u)\cup\{u\}$, and $m=0$ if there is none.
Every letter $w_i$ with $m<i<r$ is different from $u$ and satisfies $x_{w_i}x_u\notin I$, so by Lemma~\ref{lem:exchange}\eqref{lem:1words} we may move $u$ to the left past $w_{r-1},\dots,w_i$.
The result is a word of $p$ that agrees with $\omega(p)$ before position $i$ and has $u$ in position $i$, so by the maximality of $\omega(p)$ we have $w_i>u$.
Thus all the letters strictly between positions $m$ and $r$ are larger than $u$.
If $m>0$ and $w_m\in N(u)$, then $u\in S_\Lambda(w_1\cdots w_{r-1})$ by the description above.
If $m>0$ and $w_m=u\notin N(u)$, then moving $u$ to position $m+1$ gives a word of $p$ with two consecutive letters equal to $u$, that is, a path $P\xrightarrow{u}Q\xrightarrow{u}R$ with $x_u^2\notin I$, which contradicts Lemma~\ref{lem:diamond-views}\eqref{lem:diam-squares}.
If $m=0$, then moving $u$ to position $1$ gives a word of $p$ whose first letter is $u$, so there is an arrow labeled $u$ into the vertex of degree $0$.
The arrows into this vertex are those of the initial map, so $u\in\Lambda$, and since $w_1,\dots,w_{r-1}$ are all larger than $u$ we again get $u\in S_\Lambda(w_1\cdots w_{r-1})$.
Hence $\omega(p)$ is valid.

Conversely, let $W=w_1\cdots w_k$ be a valid word.
We show by induction on $k$ that $W$ is a word of some vertex $t$, and then that $\omega(t)=W$.
The case $k=0$ is trivial, so let $k\geq 1$ and $u=w_k$.
By induction, $w_1\cdots w_{k-1}$ is the word of a path $q=z_{k-1}\to\cdots\to z_0$, and it suffices to find an arrow labeled $u$ into $q$.
Let $j$ be the last position with $w_j\in N(u)$, and $j=0$ if there is none.
Since $u\in S_\Lambda(w_1\cdots w_{k-1})$, the description above shows that $w_{j+1},\dots,w_{k-1}$ are all larger than $u$, and that $u\in\Lambda$ when $j=0$.
There is an arrow labeled $u$ into $z_j$: for $j>0$ it is given by the column completion on $z_j\xrightarrow{w_j}z_{j-1}$, since $x_ux_{w_j}\in I$, and for $j=0$ it is the arrow labeled $u$ of the initial map.
Now suppose there is an arrow $y\xrightarrow{u}z_i$ for some $j\leq i<k-1$.
The letter $w_{i+1}$ is larger than $u$ and, by the choice of $j$, not in $N(u)$, so the diamond completion on $z_{i+1}\xrightarrow{w_{i+1}}z_i$ and $y\xrightarrow{u}z_i$ gives an arrow labeled $u$ into $z_{i+1}$.
By induction on $i$ there is an arrow $t\xrightarrow{u}q$, and $W$ is a word of $t$.

Finally, suppose that $\omega(t)=w'_1\cdots w'_k$ is different from $W$.
Then $\omega(t)>W$, so there is a position $i$ with $w'_1\cdots w'_{i-1}=w_1\cdots w_{i-1}$ and $a=w'_i>w_i$.
By Lemma~\ref{lem:exchange}\eqref{lem:2words}, $\omega(t)$ is obtained from $W$ by exchanges of letters $a'\neq b'$ with $x_{a'}x_{b'}\notin I$.
Such exchanges do not change the relative order of the occurrences of $a$ and of any letter $b\in N(a)\cup\{a\}$, so this order is the same in $W$ and in $\omega(t)$.
Let $c$ be the number of occurrences of $a$ in $w_1\cdots w_{i-1}$.
In $\omega(t)$ the occurrence number $c+1$ of $a$ is in position $i$, so the letters of $N(a)\cup\{a\}$ preceding it all lie among $w_1,\dots,w_{i-1}$.
Hence in $W$ the occurrence number $c+1$ of $a$ is in some position $r>i$, and none of $w_i,\dots,w_{r-1}$ lies in $N(a)\cup\{a\}$.
Since $W$ is valid, $a=w_r\in S_\Lambda(w_1\cdots w_{r-1})$, so by the description above either $a\in N(w_j)$ for some $j$ such that the letters after $w_j$ are larger than $a$, and necessarily $j<i$, or $a\in\Lambda$ and $w_1,\dots,w_{r-1}$ are all larger than $a$.
In both cases $w_i>a$, a contradiction.
Hence $\omega(t)=W$.
\end{proof}

\subsection{Formula via linear algebra}

For $k \geq 0$ let $c_k(\Lambda,\mathsf{G})$ be the number of $(\mathsf{G},\Lambda)$-valid words of length $k$; in particular $c_0(\Lambda,\mathsf{G}) = 1$, since the empty word is valid.
We want to study the generating function $$\F(\Lambda,\mathsf{G};z) = \sum_{k \geq 0} c_k(\Lambda,\mathsf{G})\, z^k \in \mathbb{Z}[[z]],$$ and we write $\F(a,\mathsf{G};z)$ when $\Lambda = \{a\}$.
Its interest comes from the following proposition.

\begin{proposition}\label{prop:poincare_words}
Let $R$ be a quadratic monomial ring with graph $\mathsf{G}$, let $\Lambda\subseteq[n]$ be nonempty, and let $M=R/(x_\Lambda)$.
Then for every $k\geq 0$ the basis elements of $F_k$ in the Column-Diamond resolution of $M$ are indexed by the $(\mathsf{G},\Lambda)$-valid words of length $k$, through $e_p\mapsto\omega(p)$.
In particular $\beta_k(M)=c_k(\Lambda,\mathsf{G})$ for every $k\geq 0$, that is, $\mathsf{P}(M;z)=\F(\Lambda,\mathsf{G};z)$.
\end{proposition}

\begin{proof}
By Remark~\ref{rem:minimal}, $\beta_k(M)$ is the number of vertices of degree $k$ of the Column-Diamond diagram of $M$, and by Lemma~\ref{lem:vertex-words} this number is $c_k(\Lambda,\mathsf{G})$.
\end{proof}

\begin{example}\label{ex:path_3}
We return to the running example, so $\mathsf{G}$ has edges $\{1,2\}$ and $\{1,3\}$, and $\Lambda=\{1\}$.
Since $S_\Lambda(Wv)=N(v)\cup\{u\in S_\Lambda(W):u<v\}$ depends only on $S_\Lambda(W)$ and $v$, we can record the valid words in a directed graph.
Its vertices are the sets $S_\Lambda(W)$ with $W$ valid, and each such set $S$ has an arrow $S\xrightarrow{v}N(v)\cup\{u\in S:u<v\}$ for every $v\in S$.
The valid words of length $k$ are then the walks of length $k$ starting at $S_\Lambda(\emptyset)=\{1\}$.
Here only three sets occur, and the directed graph and its adjacency matrix $T$, with rows and columns indexed by these sets, are the following.
\[
\begin{tikzpicture}[baseline=(current bounding box.center),
  state/.style={draw, rounded corners, inner sep=3pt, font=\footnotesize}]
\node[state] (A) at (0,0) {$\{1\}$};
\node[state] (B) at (1.7,1.5) {$\{2,3\}$};
\node[state] (C) at (3.4,0) {$\{1,2\}$};
\draw[arr] (-0.9,0) -- (A);
\draw[arr] (A) to[bend left=15] node[above left, font=\footnotesize]{$1$} (B);
\draw[arr] (B) to[bend left=15] node[below right, font=\footnotesize]{$2$} (A);
\draw[arr] (B) to[bend left=15] node[above right, font=\footnotesize]{$3$} (C);
\draw[arr] (C) to[bend left=15] node[below left, font=\footnotesize]{$1$} (B);
\draw[arr] (C) -- node[below, font=\footnotesize]{$2$} (A);
\matrix (T) at (7.8,0.6) [matrix of math nodes, ampersand replacement=\&,
  left delimiter={(}, right delimiter={)},
  nodes={minimum width=2.4em, font=\footnotesize}, row sep=1pt]
  {0 \& 1 \& 0 \\ 1 \& 0 \& 1 \\ 1 \& 1 \& 0 \\};
\foreach \s [count=\i] in {{\{1\}},{\{2,3\}},{\{1,2\}}} {
  \node[font=\scriptsize, above=2pt] at (T-1-\i |- T.north) {$\s$};
  \node[font=\scriptsize, anchor=east] (r\i) at ([xshift=-8pt]T.west |- T-\i-1) {$\s$};
}
\node[anchor=east] at ([xshift=-3pt]r2.west) {$T=$};
\end{tikzpicture}
\]
Hence $c_k(1,\mathsf{G})$ is the sum of the first row of $T^k$, and $\F(1,\mathsf{G};z)$ is the sum of the first row of $(\mathrm{Id}-zT)^{-1}=\sum_{k\geq 0}T^kz^k$, as in the transfer-matrix method \cite[Section 4.7]{stanely2011enumerative}.
The inverse can be found using Cramer's rule.
Since $\det(\mathrm{Id}-zT)=1-2z^2-z^3=(1+z)(1-z-z^2)$, the inverse is
\[
(\mathrm{Id}-zT)^{-1}=\frac{1}{(1+z)(1-z-z^2)}\begin{pmatrix} 1-z^2 & z & z^2 \\ z+z^2 & 1 & z \\ z+z^2 & z+z^2 & 1-z^2 \end{pmatrix}.
\]
The first row sums to $(1+z)/\bigl((1+z)(1-z-z^2)\bigr)$, hence $\F(1,\mathsf{G};z)=1/(1-z-z^2)$.
By Proposition~\ref{prop:poincare_words} this is the Poincaré series of $R/(x_1)$, whose coefficients $1,1,2,3,5,8,\dots$ are the Fibonacci numbers, in agreement with Example~\ref{ex:free-res-from-diag}.
{Thus $c_0(1,\mathsf G)=c_1(1,\mathsf G)=1$ and $c_k(1,\mathsf G)=c_{k-1}(1,\mathsf G)+c_{k-2}(1,\mathsf G)$
for $k\geq 2$.}

In this example, for nonempty valid words $W$, the set $S_\Lambda(W)$ is determined by the last letter of $W$, but this fails in general.
For instance, in Example~\ref{ex:closed-form-S} the valid words $12345$ and $15$ both end in $5$, while $S_\Lambda(12345)=\{1,2,3,4\}$ and $S_\Lambda(15)=\{1,2,4\}$, so that $123453$ is valid but $153$ is not.
\hfill$\diamond$
\end{example}

We now prove that $\F(\Lambda,\mathsf{G};z)$ is rational, by the argument already used in Example~\ref{ex:path_3}.

\begin{theorem}\label{thm:words_generating_function}
Let $\mathsf{G}$ be a finite graph on $V = \{1,\dots,n\}$, possibly with loops, and let $\Lambda \subseteq V$ be nonempty.
Then $\F(\Lambda,\mathsf{G};z)$ is a rational function, which can be written with a denominator of degree at most the number of sets of the form $S_\Lambda(W)$ with $W$ a valid word.
\end{theorem}

\begin{proof}
Let $\mathcal{S}$ be the collection of the sets $S_\Lambda(W)$ with $W$ a valid word, which is finite since its elements are subsets of $V$, and which contains $\Lambda = S_\Lambda(\emptyset)$.
Let $T$ be the matrix with rows and columns indexed by $\mathcal{S}$ whose entry $T_{S,S'}$ is the number of letters $v \in S$ with $N(v) \cup \{u \in S : u < v\} = S'$, that is, the adjacency matrix of the directed graph of Example~\ref{ex:path_3}.
If $W$ is a valid word with $S_\Lambda(W) = S$, then $Wv$ is valid exactly when $v \in S$, and in that case $S_\Lambda(Wv)$ is the set $S'$ above, which again lies in $\mathcal{S}$.
By induction on $k$, the number of valid words $W$ of length $k$ with $S_\Lambda(W) = S'$ is therefore the entry of $T^k$ in position $(\Lambda, S')$.
Summing over $S'$, $c_k(\Lambda,\mathsf{G})$ is the sum of the row of $T^k$ indexed by $\Lambda$.
Hence $\F(\Lambda,\mathsf{G};z)$ is the sum of the row indexed by $\Lambda$ of $\sum_{k \geq 0} T^k z^k = (\mathrm{Id} - zT)^{-1}$, an identity of matrices over $\mathbb{Z}[[z]]$ \cite[Section 4.7]{stanely2011enumerative}.
By Cramer's rule, every entry of $(\mathrm{Id} - zT)^{-1}$ is a polynomial divided by $\det(\mathrm{Id} - zT)$, which is a polynomial of degree at most $|\mathcal{S}|$ with constant term $1$.
\end{proof}

We can now prove the main theorem of this section.

\begin{proof}[Proof of Theorem~\ref{thm:rational_generating_function}]
By Proposition~\ref{prop:poincare_words}, $\mathsf{P}(R/(x_\Lambda);z) = \F(\Lambda,\mathsf{G};z)$, which is rational by Theorem~\ref{thm:words_generating_function}.
\end{proof}

\begin{example}\label{ex:star}
Let $K_{1,n-1}$ be the star on $\{1,\dots,n\}$, with $n \geq 2$ and edges $\{1,v\}$ for $v = 2,\dots,n$, and let $\Lambda = \{1\}$.
Then $N(1) = \{2,\dots,n\}$ and $N(v) = \{1\}$ for $v \geq 2$.
By the recursive definition, $S_\Lambda(W) = \{2,\dots,n\}$ when $W$ ends in $1$, and appending a letter $v \geq 2$ to such a word, and then letters $u \geq 2$ in decreasing order, gives the set $\{1,\dots,u-1\}$, where $u$ is the last letter.
Thus the nonempty valid words are the words $1B_1\,1B_2 \cdots 1B_r$ with $r \geq 1$, where each $B_i$ is a strictly decreasing word in $\{2,\dots,n\}$, nonempty except possibly for $B_r$.
Such decreasing words correspond to subsets of $\{2,\dots,n\}$, so the nonempty ones are counted by $h(z) = (1+z)^{n-1} - 1$.
Each factor $1B_i$ with $i<r$ contributes $zh(z)$ and the last one contributes $z(1+h(z))$, hence
\[
\F(1,K_{1,n-1};z) = 1+\frac{z(1+h)}{1-zh} = \frac{1+z}{1+z-z(1+z)^{n-1}} = \frac{1}{1-z(1+z)^{n-2}}.
\]
For $n = 3$ the star is the graph of the running example, and we recover the series $1/(1-z-z^2)$ of Example~\ref{ex:path_3}.
The same formula follows from Remark~\ref{rem:hilbert}, since $H_R(t)=\frac{1}{(1-t)^{n-1}}+\frac{t}{1-t}$ by \cite[Exercise~6.6.6]{Villarreal2001} and $H_M(t)=\frac{1}{(1-t)^{n-1}}$.
\hfill$\diamond$
\end{example}

\begin{example}\label{ex:cycles}
For $n \geq 3$ let $C_n$ be the cycle with vertices $1,\dots,n$ in cyclic order and edges $\{i,i+1\}$, indices modulo $n$, and take $\Lambda = \{1\}$.
By Proposition~\ref{prop:poincare_words}, $\F(1,C_n;z)$ is the Poincaré series of $R/(x_1)$, where $R = k[x_1,\dots,x_n]/(x_1x_2, \dots, x_{n-1}x_n, x_nx_1)$.
The following figure lists these series for $n = 3,4,5,6$ in reduced form, together with the degrees of their denominators.
\[
\begin{tikzpicture}[baseline=(current bounding box.center)]
\foreach \n/\x/\f/\d in {
  3/0/{\dfrac{1-z}{1-2z}}/1,
  4/3.6/{\dfrac{1-z-z^2}{1-2z-z^2}}/2,
  5/7.2/{\dfrac{1-2z}{1-3z+z^2}}/2,
  6/10.8/{\dfrac{1-2z-z^2+z^3}{1-3z+2z^3}}/3} {
  \begin{scope}[shift={(\x,0)}]
  \foreach \i in {1,...,\n} {
    \pgfmathsetmacro{\a}{90-360*(\i-1)/\n}
    \node[dot, minimum size=4pt, outer sep=0pt, label={[font=\scriptsize, label distance=1pt]\a:$\i$}] (v\i) at (\a:0.6) {};
  }
  \foreach \i in {1,...,\n} {
    \pgfmathtruncatemacro{\j}{mod(\i,\n)+1}
    \draw[line width=0.6pt] (v\i) -- (v\j);
  }
  \node[draw, circle, minimum size=9pt, inner sep=0pt] at (v1) {};
  \node at (0,-1.35) {$C_{\n}$};
  \node at (0,-2.35) {$\f$};
  \node[font=\small] at (0,-3.25) {degree $\d$};
  \end{scope}
}
\end{tikzpicture}
\]
For $n = 3,4,5$, apart from the constant term, the coefficients are the powers of $2$, the Pell numbers, and the Fibonacci numbers of odd index, respectively.
For instance, for $C_4$ they satisfy $a_k = 2a_{k-1} + a_{k-2}$ for $k \geq 3$, with $a_1 = 1$ and $a_2 = 2$.
The degrees behave regularly: for $n = 3,4,\dots,13$ they are $1, 2, 2, 3, 3, 4, 4, 5, 5, 6, 6$, that is, $\lfloor n/2 \rfloor$.
\hfill$\diamond$
\end{example}

\begin{example}\label{ex:paths}
The degree of the reduced denominator of $\F(\Lambda,\mathsf{G};z)$ is hard to predict, even for very simple graphs.
For $n \geq 2$ let $P_n$ be the path with vertices $1,\dots,n$ from left to right and edges $\{i,i+1\}$, and take $\Lambda = \{1\}$, the left endpoint.
By Proposition~\ref{prop:poincare_words}, $\F(1,P_n;z)$ is the Poincaré series of $R/(x_1)$, where $R = k[x_1,\dots,x_n]/(x_1x_2, x_2x_3, \dots, x_{n-1}x_n)$.
The following figure lists these series for $n = 3,4,5,6$ in reduced form, together with the degrees of their denominators.
\[
\begin{tikzpicture}[baseline=(current bounding box.center), x=0.6cm, y=1.1cm]
\foreach \n/\y in {3/0, 4/-1, 5/-2, 6/-3} {
  \node[anchor=east, font=\small] at (-0.7,\y) {$P_{\n}$};
  \foreach \i in {1,...,\n} {
    \node[dot, minimum size=4pt, label={[font=\scriptsize, label distance=-1pt]below:$\i$}] (v\i) at (\i-1,\y) {};
  }
  \pgfmathtruncatemacro{\m}{\n-1}
  \foreach \i [evaluate=\i as \j using int(\i+1)] in {1,...,\m} {
    \draw[line width=0.6pt] (v\i) -- (v\j);
  }
  \node[draw, circle, minimum size=9pt, inner sep=0pt] at (0,\y) {};
}
\node[anchor=west, font=\small] at (7,0.75) {$\F(1,P_n;z)$};
\node[font=\small] at (16,0.75) {degree};
\node[anchor=west] at (7,0)  {$\dfrac{1-z^2}{1-z-z^2}$};
\node[anchor=west] at (7,-1) {$\dfrac{1-z-z^2}{1-2z}$};
\node[anchor=west] at (7,-2) {$\dfrac{1-z-2z^2}{1-2z-z^2+z^3}$};
\node[anchor=west] at (7,-3) {$\dfrac{1-2z-z^2+z^3}{1-3z+z^2+z^3}$};
\node at (16,0)  {$2$};
\node at (16,-1) {$1$};
\node at (16,-2) {$3$};
\node at (16,-3) {$3$};
\end{tikzpicture}
\]
For each $n$ the sets $S_\Lambda(W)$ are $\{1\}$, $\{2\}$, $\{1,\dots,j-2,j\}$ for $3 \leq j \leq n$, and $\{1,\dots,n-1\}$, so Theorem~\ref{thm:words_generating_function} bounds the degree by $n+1$.
The actual degrees $2, 1, 3, 3$ are much smaller and do not grow regularly with $n$.
For $n = 2, 3, \dots, 13$ the degrees of the reduced denominators are $1, 2, 1, 3, 3, 3, 4, 5, 4, 6, 6, 6$.
\hfill$\diamond$
\end{example}

\subsection{Multigraded Betti numbers}\label{ssec:multigraded}

Every arrow of the Column-Diamond diagram is a multiplication by a variable, and this makes the resolution $\Z^n$-graded.
Let $\bm{t}=(t_1,\dots,t_n)$ be new variables, and for $\alpha\in\N^n$ write $\bm{t}^\alpha=t_1^{\alpha_1}\cdots t_n^{\alpha_n}$.
For a word $W=w_1\cdots w_k$ we write $\bm{t}^W=t_{w_1}\cdots t_{w_k}$, and we call its exponent vector the \emph{content} of $W$.
Exchanging two letters does not change the content, so by Lemma~\ref{lem:exchange}\eqref{lem:2words} all the words of a vertex $p$ have the same content, and we let $\deg e_p\in\N^n$ be the content of $\omega(p)$.
If $p\xrightarrow{j}q$ is an arrow, then
$\deg e_p=\deg e_q+\mathbf e_j$,
where $\mathbf e_j$ is the $j$-th standard basis
vector of $\Z^n$.
Thus the differentials preserve the multigrading.
The Column-Diamond resolution is a minimal $\Z^n$-graded free resolution of $M$, with $F_i=\bigoplus_{p\in\mathcal V_i}R(-\deg e_p)$.

The multigraded Betti numbers of $M$ are $\beta_{i,\alpha}(M)=\dim_k\operatorname{Tor}^R_i(M,k)_\alpha$ for $\alpha\in\N^n$, and we collect them in the multigraded Poincar\'e series $\mathsf P(M;\bm{t})=\sum_{i,\alpha}\beta_{i,\alpha}(M)\,\bm{t}^\alpha$.
Since the resolution is linear, $\beta_{i,\alpha}(M)=0$ unless $|\alpha|=\alpha_1+\cdots+\alpha_n$ equals $i$, so no separate variable is needed for the homological degree, and setting $t_1=\cdots=t_n=z$ recovers $\mathsf P(M;z)$.

\begin{theorem}\label{thm:multigraded}
Let $R$ be a quadratic monomial ring with graph $\mathsf G$, let $\Lambda\subseteq[n]$ be nonempty, and let $M=R/(x_\Lambda)$.
Then $\beta_{i,\alpha}(M)$ is the number of $(\mathsf G,\Lambda)$-valid words of length $i$ and content $\alpha$, that is, $\mathsf P(M;\bm{t})=\sum_W \bm{t}^W$, where $W$ runs over the valid words.
Moreover, $\mathsf P(M;\bm{t})$ is a rational function that can be written with denominator $\det(\mathrm{Id}-T(\bm{t}))$, where $T(\bm{t})$ is the matrix indexed by the sets $S_\Lambda(W)$ whose entry $T(\bm{t})_{S,S'}$ is the sum of the variables $t_v$ over the letters $v\in S$ with $N(v)\cup\{u\in S:u<v\}=S'$.
\end{theorem}

\begin{proof}
The minimal $\Z^n$-graded resolution above has $\beta_{i,\alpha}(M)$ basis elements of degree $\alpha$ in $F_i$, and by Lemma~\ref{lem:vertex-words} these are indexed by the valid words $\omega(p)$ of length $i$ and content $\alpha$.
For the second statement, the proof of Theorem~\ref{thm:words_generating_function} applies verbatim if each word $W$ is counted with weight $\bm{t}^W$ instead of $1$.
The sum of $\bm{t}^W$ over the valid words $W$ of length $k$ with $S_\Lambda(W)=S'$ is the entry of $T(\bm{t})^k$ in position $(\Lambda,S')$, so $\mathsf P(M;\bm{t})$ is the sum of the row indexed by $\Lambda$ of $(\mathrm{Id}-T(\bm{t}))^{-1}$, and we conclude by Cramer's rule.
\end{proof}

\begin{example}\label{ex:path_3_multigraded}
In the running example, with the sets ordered as $\{1\},\{2,3\},\{1,2\}$ as in Example~\ref{ex:path_3}, we replace each $1$ of $T$ by the variable $t_j$ corresponding to the label $j$ of the arrow and get
\[
T(\bm{t})=\begin{pmatrix} 0 & t_1 & 0 \\ t_2 & 0 & t_3 \\ t_2 & t_1 & 0 \end{pmatrix},
\qquad
\det(\mathrm{Id}-T(\bm{t}))=1-t_1t_2-t_1t_3-t_1t_2t_3.
\]
Writing $P_S$ for the sum of $\bm{t}^W$ over the walks $W$ starting at the set $S$, the vector $(P_S)_S$ satisfies $(\mathrm{Id}-T(\bm{t}))(P_S)_S=(1,1,1)^{\top}$, that is, $P_{\{1\}}=1+t_1P_{\{2,3\}}$, $P_{\{2,3\}}=1+t_2P_{\{1\}}+t_3P_{\{1,2\}}$ and $P_{\{1,2\}}=1+t_1P_{\{2,3\}}+t_2P_{\{1\}}$.
The first and third equations give $P_{\{1,2\}}=(1+t_2)P_{\{1\}}$, and solving,
\[
\mathsf P(R/(x_1);\bm{t})=\frac{1+t_1}{1-t_1t_2-t_1t_3-t_1t_2t_3}.
\]
The denominator is irreducible and prime to the numerator.
Setting $t_1=t_2=t_3=z$ gives $(1+z)/(1-2z^2-z^3)$, so the common factor $1+z$ found in Example~\ref{ex:path_3} only appears after specializing.
The terms of degree three of this series are $t_1^2t_2+t_1^2t_3+t_1t_2t_3$, the contents of the valid words $121$, $131$ and $132$.
In degree four the coefficient of $t_1^2t_2t_3$ is $3$, coming from the valid words $1213$, $1312$ and $1321$, so the multigraded Betti numbers can be larger than one.
\hfill$\diamond$
\end{example}

The multigraded Poincar\'e series can also be computed from Hilbert series.
For a graph $\mathsf G$, possibly with loops, let $H_{\mathsf G}(\bm{t})=\sum_{\alpha}\dim_k R_\alpha\,\bm{t}^\alpha$ be the multigraded Hilbert series of its quadratic monomial ring $R$. Here $R_\alpha$ is the $k$-span of the class of $x^\alpha=x_1^{\alpha_1}\cdots x_n^{\alpha_n}$ in $R$, which is $k$ if $x^\alpha\notin I$ and $0$ otherwise.
\begin{remark}\label{rem:hilbert-graph}
The series $H_{\mathsf G}(\bm t)$ can be read off from the graph.
Since $I$ is generated by monomials, the monomials not in $I$ form a $k$-basis of $R$ \cite[Proposition~2.4.12]{Villarreal2001}.
A monomial $x_1^{a_1}\cdots x_n^{a_n}$ is not in $I$ exactly when its support $A=\{i:a_i>0\}$ contains no edge $\{i,j\}$ of $\mathsf G$ with $i\neq j$, and $a_i=1$ for every $i\in A$ with a loop.
For each $i$ let $h_i$ be the sum of the allowed powers of $t_i$, that is, $h_i=t_i$ if there is a loop at $i$ and $h_i=t_i+t_i^2+\cdots=t_i/(1-t_i)$ otherwise.
Choosing first the support and then the exponents gives $H_{\mathsf G}(\bm t)=\sum_A\prod_{i\in A}h_i$, where $A$ runs over the subsets of $[n]$ with no edge between two distinct elements.
Separating the sets $A$ that contain a vertex $v$ from those that do not gives the recursion $H_{\mathsf G}(\bm t)=H_{\mathsf G\setminus\{v\}}(\bm t)+h_v\,H_{\mathsf G\setminus(N(v)\cup\{v\})}(\bm t)$, with $H_\emptyset=1$.
If $\mathsf G$ has no loops, the sets $A$ are the independent sets of $\mathsf G$, that is, the faces of the Stanley--Reisner complex of $I$ \cite[Section~6.2]{Villarreal2001}.
Setting $t_1=\cdots=t_n=z$ then recovers the formula $\mathsf H(R;z)=\sum_{i\geq 0}f_{i-1}\,z^i/(1-z)^i$ of \cite[Section~6.6]{Villarreal2001}, where $f_{i-1}$ is the number of independent sets with $i$ elements, and the recursion becomes \cite[Exercise~6.6.5]{Villarreal2001}.
In the running example the sets $A$ are $\emptyset,\{1\},\{2\},\{3\},\{2,3\}$, so $H_{\mathsf G}(\bm t)=\frac{1}{(1-t_2)(1-t_3)}+\frac{t_1}{1-t_1}$ and $H_{\mathsf G\setminus\{1\}}(\bm t)=\frac{1}{(1-t_2)(1-t_3)}$.
\end{remark}

Since $R/(x_\Lambda)$ is the quadratic monomial ring of the induced subgraph $\mathsf G\setminus\Lambda$ on $[n]\setminus\Lambda$, the analogue of \eqref{eq:hilbert-poincare} is
\begin{equation}\label{eq:multigraded-hilbert}
\mathsf P(M;\bm{t})\,H_{\mathsf G}(-\bm{t})=H_{\mathsf G\setminus\Lambda}(-\bm{t}).
\end{equation}

When every vertex of $\mathsf G$ has a neighbour we may take $\Lambda=[n]$, so that $M=k$ and $H_{\mathsf G\setminus\Lambda}=1$, and \eqref{eq:multigraded-hilbert} becomes the multigraded form of Fr\"oberg's relation $\mathsf P(k;\bm{t})\,H_{\mathsf G}(-\bm{t})=1$ \cite{Froberg}.

Together with Theorem~\ref{thm:multigraded}, \eqref{eq:multigraded-hilbert} gives the identity
\[
\Bigl(\sum_{W}\bm{t}^W\Bigr)\,H_{\mathsf G}(-\bm{t})=H_{\mathsf G\setminus\Lambda}(-\bm{t}),
\]
where $W$ runs over the $(\mathsf G,\Lambda)$-valid words, which relates valid words to sets with no edges between their elements and involves no algebra.
Our proof of it goes through the Column-Diamond resolution.

\begin{question}
Is there a direct combinatorial proof of this identity, for instance by a sign-reversing involution?
\end{question}

Berglund, Blasiak and Hersh \cite{berglund-2} also study multigraded Poincar\'e series over monomial rings, for the residue field $k$.
When $I$ is generated in degree two, they write the denominator of the multigraded Poincar\'e series of $k$ in terms of the M\"obius function of the lcm lattice $L_I$ of the minimal generators of $I$ \cite[Corollary~5.2]{berglund-2}.
Their series has variables $x_1,\dots,x_n$ for the multidegree and $z$ for the homological degree, but by linearity only the products $x_iz$ occur, and we identify them with $t_i$.
Taking $\Lambda=[n]$, so that $M=k$, Theorem~\ref{thm:multigraded} gives a second denominator, $\det(\mathrm{Id}-T(\bm t))$.

The two expressions are of a rather different nature.
Theirs is a sum over the elements of $L_I$, which do not depend on the order of the variables, while ours is the determinant of a matrix indexed by the sets $S_\Lambda(W)$, which do.
It would be interesting to understand the relation between the two formulas.
For instance, whenever the two denominators coincide, expanding the determinant expresses the M\"obius function of $L_I$ in terms of the transfer matrix, and determining this M\"obius function for $I$ generated in degree two is \cite[Question~5.3]{berglund-2}.

\section{An application to injective dimension}\label{sec:injdim}

We use the Column-Diamond resolution $F$ of $M=R/(x_1)$ to compute $\Ext^i_R(M,R)$.
{Applying $\Hom_R(-,R)$ to $F$ gives a cochain
complex whose term in cohomological degree $i$
is $\Hom_R(F_i,R)$, with dual basis
$\{e_p^*:p\in\mathcal V_i\}$.
Each arrow is reversed, keeping its label and sign.}
That is, the differential of $\Hom_R(F,R)$ is given by
\[
\delta e^*_q=\sum_{p \xrightarrow{\pm x_l}q}\pm\,x_l\,e^*_p ,
\]
the sum running over the arrows of $F$ that go into $q$. Hence, $\Ext^i_R(M,R)$ is the cohomology of this resulting complex in degree $i$.
Recall that $\id_R R\leq d$ if and only if
$\Ext_R^i(X,R)=0$ for every $R$-module $X$
and every $i>d$.
Thus $\id_R R=\infty$ if
$\Ext_R^i(M,R)\neq0$ for infinitely many $i$.

\begin{proposition}
Let $n \geq 3$, consider the ring $R=k[x_1,\dots,x_n]/(x_1x_2, x_1x_3 ,\dots,x_1x_n)$ and $M=R/(x_1)$. Then \( \Ext^{m(n-1)+1}_R(M,R)\neq0 \) for every $m \geq 1$. In particular $\id_RR=\infty$.
\end{proposition}

\begin{proof}
Let $W$ be the word $1,n,\dots,3$.
By Lemma~\ref{lem:closed-form-S}, we have that $S_\Lambda(W)=\{1,2\}$, and since $S_\Lambda(W'1)=N(1)=\{2,\dots,n\}$ for every word $W'$, we also have $S_\Lambda(W^m)=\{1,2\}$ for every $m\geq1$ (here $W^m$ is the concatenation of $W$ with itself $m$ times).
{In particular, $W^m$, $W^m1$ and $W^m2$
are valid words for every $m\geq1$.}
Fix $m\geq1$, let $d=m(n-1)$, and let $u$, $w$ and $w_2$ be the vertices whose words are $W^m$, $W^m1$ and $W^m2$, which exist by Lemma~\ref{lem:vertex-words}.
They have degrees $d$, $d+1$ and $d+1$, and there are arrows $w\xrightarrow{1}u$ and $w_2\xrightarrow{2}u$.

Lemma~\ref{lem:diamond-views}\eqref{lem:diam-squares} implies that no arrow labeled $1$ goes into $w$, since $x_1^2\notin I$.
Hence every arrow into $w$ is labeled by some $l\neq1$, and since $x_1x_l\in I$ we get $\delta(x_1e^*_w)=0$.

By the same token, $w\xrightarrow{1}u$ is the only arrow out of $w$.
Indeed, $W^m1$ is a word of $w$, as in Definition~\ref{def:word-of-vertex}, and since $x_1x_l\in I$ for every $l\neq1$, no exchange as in Lemma~\ref{lem:exchange} moves the letter $1$ in any word.
Hence, by Lemma~\ref{lem:exchange}\eqref{lem:2words}, every path from $w$ to the vertex of degree $0$ starts with an arrow labeled $1$.
Thus every arrow out of $w$ is labeled $1$, and by Lemma~\ref{lem:diamond-views}\eqref{lem:diam-upperarrows} there is only one such arrow.
Let $\sigma$ be the sign of $w\xrightarrow{1}u$ and $\tau$ the sign of $w_2\xrightarrow{2}u$.
The relevant part of the Column-Diamond diagram is the following.
\begin{equation}\label{fig:star-branch}
\begin{tikzpicture}[baseline=(current bounding box.center)]
\draw[gray!60, dashed] (-3.2,4.2) -- (3.9,4.2);
\draw[gray!60, dashed] (-3.2,5.0) -- (3.9,5.0);
\node[font=\footnotesize] at (4.3,4.2) {$d$};
\node[font=\footnotesize] at (4.3,5.0) {$d+1$};

\node[dot] (o)  at (0,0) {};
\node[dot] (v0) at (0,1.4) {};
\node[dot] (v1) at (0,2.2) {};
\node[dot] (v2) at (0,3.4) {};
\node[dot, label={[font=\footnotesize, fill=white, inner sep=1pt]left:$u$}] (u) at (0,4.2) {};
\node[dot, label={[font=\footnotesize, fill=white, inner sep=1pt]left:$w$}] (w) at (-1.4,5.0) {};
\node[dot, label={[font=\footnotesize, fill=white, inner sep=1pt]right:$w_2$}] (w2) at (1.4,5.0) {};

\draw[decorate, decoration={snake, amplitude=1.5pt, segment length=6pt}] (v0) -- (o);
\draw[arr] (v1) -- node[right, font=\footnotesize]{$1$} (v0);
\draw[dotted, line width=0.6pt] (v1) -- (v2);
\draw[arr] (u) -- node[right, font=\footnotesize]{$3$} (v2);
\draw[arr] (w) -- node[left, font=\footnotesize]{$1\;$} (u);
\draw[arr] (w2) -- node[right, font=\footnotesize]{\;$2$} (u);

\node[dot] (a2) at (-2.4,5.8) {};
\node[dot] (a3) at (-1.6,5.8) {};
\node at (-0.95,5.8) {$\cdots$};
\node[dot] (an) at (-0.3,5.8) {};
\draw[arr] (a2) -- node[left, font=\footnotesize]{$2\;$} (w);
\draw[arr] (a3) -- node[right, pos=0.4, font=\footnotesize]{$3$} (w);
\draw[arr] (an) -- node[right, font=\footnotesize]{\;$n$} (w);

\node[dot] (b1) at (1.9,4.2) {};
\node[fill=white, inner sep=1pt] at (2.6,4.2) {$\cdots$};
\node[dot] (b2) at (3.3,4.2) {};
\draw[arr] (w2) -- (b1);
\draw[arr] (w2) -- (b2);

\draw[decorate, decoration={brace, amplitude=4pt, mirror}] (0.35,1.4) -- node[right=4pt, font=\footnotesize]{$W$} (0.35,4.2);
\draw[decorate, decoration={brace, amplitude=4pt, mirror}] (0.35,0) -- node[right=4pt, font=\footnotesize]{$W^{m-1}$} (0.35,1.4);
\end{tikzpicture}
\end{equation}

Now, suppose that $x_1e^*_w=\delta(y)$ for some $y=\sum_{q\in\mathcal V_d}f_qe^*_q$.
Each arrow $p\xrightarrow{l}q$ of $F$ with sign $\varepsilon$ contributes $\varepsilon x_lf_q$ to the coefficient of $e^*_p$ in $\delta(y)$.
Since $w\xrightarrow{1}u$ is the only arrow out of $w$, the coefficient of $e^*_w$ is $\sigma x_1f_u$.
It must be equal to $x_1$, so the constant term of $f_u$ is $\sigma$.
The same $f_u$ contributes $\tau x_2f_u$ to the coefficient of $e^*_{w_2}$, through the arrow $w_2\xrightarrow{2}u$.
By Lemma~\ref{lem:diamond-views}\eqref{lem:diam-upperarrows}, every other arrow out of $w_2$ has a label $l\neq2$, so it contributes $\pm x_lf_q$.
Since the constant term of $f_u$ is $\sigma$, the coefficient of $e^*_{w_2}$ has component of degree one $\tau\sigma x_2+\sum_{l\neq2}c_lx_l$ for some $c_l\in k$.
This is a nonzero linear form, so it is not zero in $R$, since $I$ is generated by quadratic monomials.
But this coefficient must be zero, since $\delta(y)=x_1e^*_w$.
This contradiction shows that $x_1e^*_w$ is not a coboundary, so $\Ext^{d+1}_R(M,R)\neq0$.
\end{proof}

\begin{remark}
The result can be read directly from the Column-Diamond construction, without computing any cohomology.
The construction is local, and after each arrow $1$ the new vertex sees the same data, since $S_\Lambda(W'1)=N(1)$ for every word $W'$.
Hence the construction produces the figure \eqref{fig:star-branch} again and again, once for each $m\geq1$.
After applying $\Hom_R(-,R)$ the arrows are reversed.
In each copy, the arrows $2,\dots,n$ into $w$ make $x_1e^*_w$ a cocycle, and the branching at $u$, where the arrows $1$ and $2$ arrive, prevents it from being a coboundary.
Thus the infinitely many nonzero classes in $\Ext_R(M,R)$ are a direct consequence of the repetition of a single figure in the Column-Diamond diagram.
\end{remark}

\begin{remark}
That $\id_RR=\infty$ is known. Since $I=(x_1)\cap(x_2,\dots,x_n)$, the spectrum of $R$ has two components, of dimensions $n-1$ and $1$. For $n\ge3$ these are different, so the local ring $R_{\mathfrak m}$, with $\mathfrak m=(x_1,\dots,x_n)$, is not equidimensional, hence not Cohen-Macaulay \cite[Theorem~2.1.2(a)]{BH} and not Gorenstein. By the theorem of Bass \cite{Bass}, $\id_{R_{\mathfrak m}}R_{\mathfrak m}=\infty$, and therefore $\id_RR=\infty$. Alternatively, $R$ is the Stanley-Reisner ring of the disjoint union of a point and an $(n-2)$-simplex, which is not Cohen-Macaulay by Reisner's criterion \cite{Reisner}, since it is disconnected of positive dimension. For $n=2$ the ring $k[x_1,x_2]/(x_1x_2)$ is a hypersurface, hence Gorenstein. What the Column-Diamond resolution adds is the explicit obstruction, namely the classes $x_1e^*_w$, where $w$ is the vertex with word $(1,n,\dots,3)^m1$, which are nonzero in $\Ext^{m(n-1)+1}_R(R/(x_1),R)$ and are read directly from the diagram.
\end{remark}

\section*{Declaration of generative AI use}
During the preparation of this work, the authors used Claude Opus (Anthropic), an AI system, as a research assistant: to discuss and check arguments, to carry out computer verifications of examples, and to edit the text. All mathematical content was verified by the authors.

\section*{Acknowledgments}

The authors thank Elisa Gorla and Jason McCullough for insightful conversations.
This collaboration began in January 2025, when DB and NV were both visiting the Pontificia Universidad Católica de Chile to work with FC.
DB and FC thank the \emph{Visitante Nacional} programme of the Pontificia Universidad Católica de Chile. They also thank the Instituto de Ciencias Físicas y Matemáticas at the Universidad Austral de Chile for its hospitality and facilities during a research visit in July 2025.
 DB acknowledges support from ANID/FONDECYT/Regular 1240253. NV acknowledges support from the Welsh Government through
a \emph{Taith Mobility Award} for her research visit.

\appendix

\section{Diamonds in the Column-Diamond diagram}\label{app:diamonds}

In this appendix we prove Lemma \ref{lem:diamond-views} and show that the signs described in Remark \ref{rem:signed-diamonds} are compatible.
We recall the lemma before giving its proof.

\medskip
\noindent\textbf{Lemma \ref{lem:diamond-views}.}\ \emph{In the Column-Diamond diagram the following hold.}
\begin{enumerate}
\item\label{lem:upperarrows} \emph{Any two distinct arrows out of the same vertex carry distinct labels and are the upper arrows of a diamond.}
\item\label{lem:bnotc} \emph{Every path $P\xrightarrow{c}Q\xrightarrow{b}R$ with {$b\neq c$ and}  $x_bx_c\notin I$ lies in a diamond. %
}
\item\label{lem:square} {There is no path $P\xrightarrow{b}Q\xrightarrow{b}R$ with $x_b^2\notin I$.}

\end{enumerate}

\begin{proof}

{We prove the three statements simultaneously, by induction
on the degree of the top vertex.}

In degree one every vertex has a single arrow, the one of the initial map \eqref{fan-like-map}, so \eqref{lem:upperarrows} holds trivially.
{In degree $2$, each vertex is obtained from a single
column or diamond completion, so it has either one
outgoing arrow or the two upper arrows of a diamond.
Thus \eqref{lem:upperarrows} also holds in degree $2$.}
If $P$ has degree two, then $Q$ has degree one and its only arrow is $Q\xrightarrow{b}R$.
The arrow $P\xrightarrow{c}Q$ was then produced by an elementary diagram on $Q\xrightarrow{b}R$, which cannot be a column since $x_bx_c\notin I$.
It is therefore a diamond, and \eqref{lem:bnotc} holds.
Since the two labels of a diamond are distinct, this also gives \eqref{lem:square} in degree two.

Assume now that \eqref{lem:upperarrows}, \eqref{lem:bnotc} and \eqref{lem:square} hold in degrees at most $n$.

We first prove \eqref{lem:bnotc} for $P$ of degree $n+1$.
The arrow $P\xrightarrow{c}Q$ was produced by an elementary diagram of $D$ on some arrow $Q\xrightarrow{k}Z$.
If $k=b$, this elementary diagram cannot be a column since $x_bx_c\notin I$, so it is a diamond on $Q\xrightarrow{b}R$ and \eqref{lem:bnotc} holds.
Suppose that $k\neq b$.
Then $Q$ has degree $n$ and two arrows, labeled $b$ and $k$, so by the induction hypothesis applied to \eqref{lem:upperarrows}, there is a vertex $S$, in degree $n-2$, with $R\xrightarrow{k}S$ and $Z\xrightarrow{b}S$, as in the following diagram

\[
\begin{tikzpicture}[baseline=(current bounding box.center)]

\node[dot] (P) at (0, 2.4) {};
\node[right] at (0.08, 2.48) {$P$};

\node[dot] (Q) at (0, 1.6) {};
\node[right] at (0.08, 1.6) {$Q$};

\node[dot] (R) at (-0.8, 0.8) {};
\node[left] at (-0.9, 0.8) {$R$};

\node[dot] (Z) at (0.8, 0.8) {};
\node[right] at (0.9, 0.8) {$Z$};

\node[dot] (S) at (0, 0) {};
\node[right] at (0.08, -0.1) {$S$};

\draw[arr] (P) -- node[left, font=\footnotesize]{$c$} (Q);
\draw[arr] (Q) -- node[left, font=\footnotesize]{$b\;\,$} (R);
\draw[arr] (Q) -- node[right, font=\footnotesize]{$\;\, k$} (Z);
\draw[arr] (R) -- node[left, font=\footnotesize]{$k$} (S);
\draw[arr] (Z) -- node[right, font=\footnotesize]{$\, b$} (S);

\end{tikzpicture}
\]
If the elementary diagram on $Q\xrightarrow{k}Z$ that produces $P\xrightarrow{c}Q$ is a diamond, its other lower arrow is an arrow $Y_0\xrightarrow{c}Z$, and $x_cx_k\notin I$.
The path $Y_0\xrightarrow{c}Z\xrightarrow{b}S$ satisfies $b\neq c$ and $x_bx_c\notin I$, and $Y_0$ has degree $n$, so by the induction hypothesis applied to \eqref{lem:bnotc} there is a vertex $Y_1$ with $Y_0\xrightarrow{b}Y_1$ and $Y_1\xrightarrow{c}S$.
The arrows $R\xrightarrow{k}S$ and $Y_1\xrightarrow{c}S$ form the lower part of a diamond, whose upper arrow labeled $c$ is an arrow $Y\xrightarrow{c}R$.
The dashed arrows are the ones obtained in this way as seen in the following diagram.

\[
\begin{tikzpicture}[baseline=(current bounding box.center)]

\node[dot] (P) at (0, 2.4) {};
\node[right] at (0.08, 2.48) {$P$};

\node[dot] (Y) at (-1.5, 1.6) {};
\node[left] at (-1.6, 1.6) {$Y$};

\node[dot] (Q) at (0, 1.6) {};
\node[right] at (0.08, 1.7) {$Q$};

\node[dot] (Y0) at (1.5, 1.6) {};
\node[right] at (1.6, 1.6) {$Y_0$};

\node[dot] (R) at (-0.75, 0.8) {};
\node[left] at (-0.85, 0.8) {$R$};

\node[dot] (Z) at (0.75, 0.8) {};
\node[right] at (0.85, 0.72) {$Z$};

\node[dot] (Y1) at (2.2, 0.8) {};
\node[right] at (2.3, 0.8) {$Y_1$};

\node[dot] (S) at (0, 0) {};
\node[right] at (0.08, -0.1) {$S$};

\draw[arr] (P) -- node[left, font=\footnotesize]{$c$} (Q);
\draw[arr] (Q) -- node[left, font=\footnotesize]{$b\;$} (R);
\draw[arr] (Q) -- node[right, font=\footnotesize]{$k$} (Z);
\draw[arr] (Y0) -- node[right, font=\footnotesize]{$c$} (Z);
\draw[arr] (R) -- node[left, font=\footnotesize]{$k$} (S);
\draw[arr] (Z) -- node[right, font=\footnotesize]{$b$} (S);

\draw[arr, dashed] (Y0) -- node[right, font=\footnotesize]{$\; b$} (Y1);
\draw[arr, dashed] (Y1) -- node[below right, font=\footnotesize]{$c$} (S);
\draw[arr, dashed] (Y) -- node[left, font=\footnotesize]{$c\;$} (R);

\node[font=\footnotesize] at (4.2, 2.4) {$n+1$};
\node[font=\footnotesize] at (4.2, 1.6) {$n$};
\node[font=\footnotesize] at (4.2, 0.8) {$n-1$};
\node[font=\footnotesize] at (4.2, 0) {$n-2$};

\end{tikzpicture}
\]
If instead the elementary diagram on $Q\xrightarrow{k}Z$ that produces $P\xrightarrow{c}Q$ is a column, then $x_cx_k\in I$, and the column completion of $R\xrightarrow{k}S$ gives an arrow $Y\xrightarrow{c}R$.

In both cases there is an arrow $Y\xrightarrow{c}R$.
The arrows $Q\xrightarrow{b}R$ and $Y\xrightarrow{c}R$ form the lower part of a diamond, since $b\neq c$ and $x_bx_c\notin I$, and its upper arrows are an arrow labeled $c$ into $Q$ and an arrow labeled $b$ into $Y$, out of the same vertex.
By the ``No repetitions'' rule the first one is $P\xrightarrow{c}Q$, so this vertex is $P$ and $P\xrightarrow{b}Y$.
Hence the path $P\xrightarrow{c}Q\xrightarrow{b}R$ lies in a diamond, which proves \eqref{lem:bnotc} in degree $n+1$.

We now prove \eqref{lem:upperarrows} for a vertex $v$ of degree $n+1$ with two arrows $v\xrightarrow{a}P_a$ and $v\xrightarrow{b}P_b$.
Each of them was produced by an elementary diagram of $D$, and the elementary diagrams producing the arrows out of $v$ are merged into $v$ by the ``No repetitions'' rule.
If both arrows belong to the same elementary diagram, this diagram has two upper arrows, so it is a diamond.
By the definition of the diamonds in $D$, then $a\neq b$, $x_ax_b\notin I$, and the diamond provides a vertex $Q$ with $P_a\xrightarrow{b}Q$ and $P_b\xrightarrow{a}Q$, so \eqref{lem:upperarrows} holds.

Otherwise the arrows $v\xrightarrow{a}P_a$ and $v\xrightarrow{b}P_b$ are joined by a chain of arrows out of $v$ in which any two consecutive arrows belong to the same elementary diagram, since this is how the ``No repetitions'' rule merges them.
A column has a single upper arrow, so the elementary diagrams joining two consecutive arrows of the chain are diamonds.
By induction on the length of the chain, it suffices to consider a third arrow $v\xrightarrow{c}P_c$ such that $\{a,c\}$ and $\{c,b\}$ are the upper arrows of two diamonds, with lower vertices $Q_{ac}$ and $Q_{bc}$.
In particular $x_ax_c\notin I$ and $x_bx_c\notin I$.

Assume that $a\neq b$.
The vertex $P_c$ has degree $n$ and two arrows $P_c\xrightarrow{a}Q_{ac}$ and $P_c\xrightarrow{b}Q_{bc}$, so by the induction hypothesis applied to \eqref{lem:upperarrows}, $x_ax_b\notin I$ and there is a vertex $R$ with $Q_{ac}\xrightarrow{b}R$ and $Q_{bc}\xrightarrow{a}R$.
\[
\begin{tikzpicture}[baseline=(current bounding box.center)]

\node[dot] (v) at (0, 2.4) {};
\node[right] at (0.08, 2.48) {$v$};

\node[dot] (Pa) at (-1.5, 1.6) {};
\node[left] at (-1.6, 1.6) {$P_a$};

\node[dot] (Pc) at (0, 1.6) {};
\node[right] at (0.08, 1.7) {$P_c$};

\node[dot] (Pb) at (1.5, 1.6) {};
\node[right] at (1.6, 1.6) {$P_b$};

\node[dot] (Qac) at (-0.8, 0.8) {};
\node[left] at (-0.9, 0.8) {$Q_{ac}$};

\node[dot] (Qbc) at (0.8, 0.8) {};
\node[right] at (0.9, 0.8) {$Q_{bc}$};

\node[dot] (R) at (0, 0) {};
\node[right] at (0.08, -0.1) {$R$};

\draw[arr] (v) -- node[left, font=\footnotesize]{$a\; \; $} (Pa);
\draw[arr] (v) -- node[right, font=\footnotesize, pos=0.4]{$c$} (Pc);
\draw[arr] (v) -- node[right, font=\footnotesize]{$\;\; b$} (Pb);

\draw[arr] (Pa) -- node[left, font=\footnotesize]{$c$} (Qac);
\draw[arr] (Pc) -- node[left, font=\footnotesize, pos=0.3]{$a$} (Qac);
\draw[arr] (Pc) -- node[right, font=\footnotesize, pos=0.3]{\;$b$} (Qbc);
\draw[arr] (Pb) -- node[right, font=\footnotesize]{$c$} (Qbc);

\draw[arr] (Qac) -- node[left, font=\footnotesize]{$b$} (R);
\draw[arr] (Qbc) -- node[right, font=\footnotesize]{$a$} (R);

\node[font=\footnotesize] at (3.0, 2.4) {$n+1$};
\node[font=\footnotesize] at (3.0, 1.6) {$n$};
\node[font=\footnotesize] at (3.0, 0.8) {$n-1$};
\node[font=\footnotesize] at (3.0, 0) {$n-2$};

\end{tikzpicture}
\]

The path $P_a\xrightarrow{c}Q_{ac}\xrightarrow{b}R$ satisfies $b\neq c$ and $x_bx_c\notin I$, and $P_a$ has degree $n$, so by the induction hypothesis applied to \eqref{lem:bnotc} there is a vertex $Y$ with $P_a\xrightarrow{b}Y$ and $Y\xrightarrow{c}R$.
In the same way, the path $P_b\xrightarrow{c}Q_{bc}\xrightarrow{a}R$ gives a vertex $Y'$ with $P_b\xrightarrow{a}Y'$ and $Y'\xrightarrow{c}R$.
Since $Y$ and $Y'$ both have an arrow labeled $c$ into $R$, the ``No repetitions'' rule gives $Y=Y'$.
Hence $P_a\xrightarrow{b}Y$ and $P_b\xrightarrow{a}Y$, and the arrows labeled $a$ and $b$ out of $v$ are the upper arrows of a diamond, as in the following diagram:
\[
\begin{tikzpicture}[baseline=(current bounding box.center)]

\node[dot] (v) at (0, 2.55) {};
\node[right] at (0.1, 2.65) {$v$};

\node[dot] (Pa) at (-1.6, 1.7) {};
\node[left] at (-1.7, 1.7) {$P_a$};

\node[dot] (Pc) at (0, 1.7) {};
\node[right] at (0.1, 1.8) {$P_c$};

\node[dot] (Pb) at (1.6, 1.7) {};
\node[right] at (1.7, 1.7) {$P_b$};

\node[dot] (Qac) at (-0.9, 0.85) {};
\node[left] at (-1.0, 0.85) {$Q_{ac}$};

\node[dot] (Y) at (0, 0.85) {};
\node[above right] at (-0.25, 0.95) {$Y$};

\node[dot] (Qbc) at (0.9, 0.85) {};
\node[right] at (1.0, 0.85) {$Q_{bc}$};

\node[dot] (R) at (0, 0) {};
\node[right] at (0.1, -0.1) {$R$};

\draw[arr] (v) -- node[left, font=\footnotesize]{$a\;$} (Pa);
\draw[arr] (v) -- node[right, font=\footnotesize, pos=0.4]{$c$} (Pc);
\draw[arr] (v) -- node[right, font=\footnotesize]{\;\;$b$} (Pb);

\draw[arr] (Pa) -- node[left, font=\footnotesize]{$c$} (Qac);
\draw[arr] (Pc) -- node[left, font=\footnotesize, pos=0.3]{$a$} (Qac);
\draw[arr] (Pc) -- node[right, font=\footnotesize, pos=0.3]{$\, b$} (Qbc);
\draw[arr] (Pb) -- node[right, font=\footnotesize]{$c$} (Qbc);

\draw[arr] (Qac) -- node[left, font=\footnotesize]{$b$} (R);
\draw[arr] (Qbc) -- node[right, font=\footnotesize]{$a$} (R);

\draw[arr, dashed] (Pa) -- (Y);
\draw[arr, dashed] (Pb) -- (Y);
\draw[arr, dashed] (Y) -- (R);

\node[font=\footnotesize] at (2.8, 2.55) {$n+1$};
\node[font=\footnotesize] at (2.8, 1.7) {$n$};
\node[font=\footnotesize] at (2.8, 0.85) {$n-1$};
\node[font=\footnotesize] at (2.8, 0) {$n-2$};

\end{tikzpicture}
\]

It remains to consider the case $a=b$, in which the two arrows out of $v$ carry the same label, as in the following diagram.
\[
\begin{tikzpicture}[baseline=(current bounding box.center)]

\node[dot] (v) at (0, 1.6) {};
\node[right] at (0.08, 1.7) {$v$};

\node[dot] (Pa) at (-1.5, 0.8) {};
\node[left] at (-1.6, 0.8) {$P_a$};

\node[dot] (Pc) at (0, 0.8) {};
\node[right] at (0.08, 0.9) {$P_c$};

\node[dot] (Pb) at (1.5, 0.8) {};
\node[right] at (1.6, 0.8) {$P_b$};

\node[dot] (Qac) at (-0.8, 0) {};
\node[left] at (-0.9, -0.1) {$Q_{ac}$};

\node[dot] (Qbc) at (0.8, 0) {};
\node[right] at (0.9, -0.1) {$Q_{bc}$};

\draw[arr] (v) -- node[left, font=\footnotesize]{$a\; $} (Pa);
\draw[arr] (v) -- node[right, font=\footnotesize, pos=0.4]{$c$} (Pc);
\draw[arr] (v) -- node[right, font=\footnotesize]{$\; a$} (Pb);

\draw[arr] (Pa) -- node[left, font=\footnotesize]{$c$} (Qac);
\draw[arr] (Pc) -- node[left, font=\footnotesize, pos=0.3]{$a$} (Qac);
\draw[arr] (Pc) -- node[right, font=\footnotesize, pos=0.3]{$a$} (Qbc);
\draw[arr] (Pb) -- node[right, font=\footnotesize]{$c$} (Qbc);

\node[font=\footnotesize] at (3.0, 1.6) {$n+1$};
\node[font=\footnotesize] at (3.0, 0.8) {$n$};
\node[font=\footnotesize] at (3.0, 0) {$n-1$};

\end{tikzpicture}
\]
The vertex $P_c$ has degree $n$ and two arrows labeled $a$, so by the induction hypothesis applied to \eqref{lem:upperarrows} they are the same arrow, and $Q_{ac}=Q_{bc}$.
Then $P_a$ and $P_b$ both have an arrow labeled $c$ into this vertex, and the ``No repetitions'' rule gives $P_a=P_b$.
Hence two distinct arrows out of $v$ carry distinct labels, which completes the induction {on the length of the chain and} proves \eqref{lem:upperarrows} in degree $n+1$.

We finally prove \eqref{lem:square} in degree $n+1$: there is no path $P\xrightarrow{b}Q\xrightarrow{b}R$ with $x_b^2\notin I$ and $P$ of degree $n+1$.
Indeed, the arrow $P\xrightarrow{b}Q$ is produced by an elementary diagram on some arrow $Q\xrightarrow{k}Z$. It cannot be a column on $Q\xrightarrow{b}R$, since $x_b^2\notin I$, nor a diamond on it, since the two labels of a diamond are distinct. So $k\neq b$, and then $x_bx_k\notin I$ by \eqref{lem:upperarrows} at $Q$, so the elementary diagram is not a column either. It is therefore a diamond with another lower arrow $Y_0\xrightarrow{b}Z$, and by \eqref{lem:upperarrows} at $Q$ there is an arrow $Z\xrightarrow{b}S$.
The path $Y_0\xrightarrow{b}Z\xrightarrow{b}S$ then contradicts \eqref{lem:square} in degree $n$.
\end{proof}

We recall the rule of Remark \ref{rem:signed-diamonds}.
The signs of a diamond are chosen so that its two paths carry opposite signs.
If the lower arrows of a diamond carry $\varepsilon x_i$ and $\delta x_j$, and its upper arrows labeled $j$ and $i$ carry the signs $\varepsilon'$ and $\delta'$, this means $\varepsilon'\varepsilon=-\delta'\delta$.
When a vertex is the top of several diamonds, the arrow out of it with the smallest label is positive, and these relations determine the remaining signs.
We now show that this rule is well defined.

\begin{lemma}\label{lem:signs-compatible}
{Let $v$ be a vertex of the Column-Diamond diagram
that is the top of one or more diamonds.
The signs of the arrows out of $v$ can be chosen
so that, in every diamond with top vertex $v$,
the two paths from its top vertex to its bottom
vertex have opposite products of signs.}
\end{lemma}

\begin{proof}[Proof of Lemma \ref{lem:signs-compatible}]
We argue by induction on the degree of $v$.
{The arrows out of vertices of degree $1$ have positive signs,
as given by the initial map.
In degree $2$, a vertex can be the top of at most one
diamond, so the signs of its upper arrows can be chosen
such that the two paths have opposite products of signs.

Now assume that signs have been chosen for the arrows
out of all vertices of lower degree, so that the two
paths of every diamond with such a top vertex have
opposite products of signs.}

By Lemma \ref{lem:diamond-views}, the diamonds with top vertex $v$ correspond to the pairs of arrows out of $v$.
Let $v\xrightarrow{a}P_a$ and $v\xrightarrow{b}P_b$ be two arrows out of $v$, with signs $\varepsilon_a$ and $\varepsilon_b$.
They are the upper arrows of a diamond, so there is a vertex $Q_{ab}$ with $P_a\xrightarrow{b}Q_{ab}$ and $P_b\xrightarrow{a}Q_{ab}$.
This way, the sign relation of Remark \ref{rem:signed-diamonds} reads
\[
\varepsilon_a\varepsilon_b=-\,\sigma(P_a\to Q_{ab})\,\sigma(P_b\to Q_{ab}),
\]
where $\sigma$ denotes the sign of an arrow.
If $v$ has at most two arrows there is at most one relation, and it can always be satisfied.

Suppose now that $v$ has at least three arrows, and let $a$ be the smallest label among them.
We put $\varepsilon_a=1$, and for every other arrow $v\xrightarrow{b}P_b$ the relation of the pair $\{a,b\}$ determines $\varepsilon_b$.
It remains to check that the relation of every pair $\{b,c\}$ not containing $a$ also holds.
Consider three arrows $v\xrightarrow{a}P_a$, $v\xrightarrow{b}P_b$, $v\xrightarrow{c}P_c$, and the three diamonds they form, with the signs indicated in the following diagram.
\[
\begin{tikzpicture}[baseline=(current bounding box.center), xscale=1.5, yscale=1.4,
  lab/.style={font=\scriptsize, inner sep=1pt}]
\node[dot] (v) at (0, 1.4) {};
\node[right] at (0, 1.5) {\ $v$};
\node[dot] (Pa) at (-1.2, 0.7) {};
\node[left] at (-1.2, 0.7) {$P_a$\ };
\node[dot] (Pb) at (0, 0.7) {};
\node[above right] at (0.05, 0.75) {$P_b$};
\node[dot] (Pc) at (1.2, 0.7) {};
\node[right] at (1.2, 0.7) {\ $P_c$};
\node[dot] (Qab) at (-1.2, 0) {};
\node[left] at (-1.2, 0) {$Q_{ab}$\ };
\node[dot] (Qac) at (0, 0) {};
\node[below] at (0, 0) {$Q_{ac}$};
\node[dot] (Qbc) at (1.2, 0) {};
\node[right] at (1.2, 0) {\ $Q_{bc}$};
\draw[arr] (v) -- node[lab, above left]{$1=\varepsilon_a$} (Pa);
\draw[arr] (v) -- node[lab, left, pos=0.4]{$\varepsilon_b$} (Pb);
\draw[arr] (v) -- node[lab, above right]{$\varepsilon_c$} (Pc);
\draw[arr] (Pa) -- node[lab, left]{$\sigma_1$} (Qab);
\draw[arr] (Pa) -- node[lab, below left, pos=0.8]{$\sigma_2$} (Qac);
\draw[arr] (Pb) -- node[lab, above left, pos=0.2]{$\sigma_3$} (Qab);
\draw[arr] (Pb) -- node[lab, above right, pos=0.2]{$\sigma_4$} (Qbc);
\draw[arr] (Pc) -- node[lab, below right, pos=0.8]{$\sigma_5$} (Qac);
\draw[arr] (Pc) -- node[lab, right]{$\sigma_6$} (Qbc);
\end{tikzpicture}
\]
The three relations read
\[
\varepsilon_a\varepsilon_b=-\sigma_1\sigma_3,\qquad
\varepsilon_a\varepsilon_c=-\sigma_2\sigma_5,\qquad
\varepsilon_b\varepsilon_c=-\sigma_4\sigma_6 .
\]
Since $\varepsilon_a=1$, the first two relations give $\varepsilon_b=-\sigma_1\sigma_3$ and $\varepsilon_c=-\sigma_2\sigma_5$, so that $\varepsilon_b\varepsilon_c=\sigma_1\sigma_2\sigma_3\sigma_5$.
The third relation requires instead $\varepsilon_b\varepsilon_c=-\sigma_4\sigma_6$.
Hence the three relations are coherent exactly when these two values coincide, that is, when $-\sigma_4\sigma_6=\sigma_1\sigma_2\sigma_3\sigma_5$.
Multiplying by $\sigma_4\sigma_6$, this condition becomes
\begin{equation}\label{eq:sign-coherence}
\sigma_1\sigma_2\sigma_3\sigma_4\sigma_5\sigma_6=-1 .
\end{equation}
{Each $P_a$, $P_b$ and $P_c$ has two outgoing arrow, }
so by Lemma \ref{lem:diamond-views} these {pairs of} arrows are the upper arrows of three diamonds.
Their lower vertices coincide.
Indeed, the diamond below $P_a$ closes with arrows $Q_{ab}\xrightarrow{c}R_a$ and $Q_{ac}\xrightarrow{b}R_a$, and the diamond below $P_b$ closes with arrows $Q_{ab}\xrightarrow{c}R_b$ and $Q_{bc}\xrightarrow{a}R_b$.
Both $Q_{ab}\xrightarrow{c}R_a$ and $Q_{ab}\xrightarrow{c}R_b$ are arrows out of $Q_{ab}$ with the same label, so they are equal by Lemma \ref{lem:diamond-views}, and $R_a=R_b$.
In the same way, the diamond below $P_c$ has the same lower vertex.
We call $R$ this common vertex, and show all this together in the following cube diagram.
\[
\begin{tikzpicture}[baseline=(current bounding box.center), xscale=1.5, yscale=1.4,
  lab/.style={font=\scriptsize, inner sep=1pt}]
\node[dot] (v) at (0, 2.1) {};
\node[right] at (0, 2.2) {\ $v$};
\node[dot] (Pa) at (-1.2, 1.4) {};
\node[left] at (-1.2, 1.5) {$P_a$\ };
\node[dot] (Pb) at (0, 1.4) {};
\node[above right] at (0.05, 1.3) {$P_b$};
\node[dot] (Pc) at (1.2, 1.4) {};
\node[right] at (1.2, 1.5) {\ $P_c$};
\node[dot] (Qab) at (-1.2, 0.7) {};
\node[left] at (-1.2, 0.6) {$Q_{ab}$\ };
\node[dot] (Qac) at (0, 0.7) {};
\node[below right] at (0.05, 0.8) {$Q_{ac}$};
\node[dot] (Qbc) at (1.2, 0.7) {};
\node[right] at (1.2, 0.6) {\ $Q_{bc}$};
\node[dot] (R) at (0, 0) {};
\node[right] at (0, -0.1) {\ $R$};
\draw[arr] (v) -- node[lab, above left]{$a$} (Pa);
\draw[arr] (v) -- node[lab, left, pos=0.4]{$b$} (Pb);
\draw[arr] (v) -- node[lab, above right]{$c$} (Pc);
\draw[arr] (Pa) -- node[lab, left]{$b$} (Qab);
\draw[arr] (Pa) -- node[lab, above right, pos=0.2]{$c$} (Qac);
\draw[arr] (Pb) -- node[lab, above left, pos=0.2]{$a$} (Qab);
\draw[arr] (Pb) -- node[lab, below left, pos=0.2]{$c$} (Qbc);
\draw[arr] (Pc) -- node[lab, above left, pos=0.2]{$a$} (Qac);
\draw[arr] (Pc) -- node[lab, right]{$b$} (Qbc);
\draw[arr] (Qab) -- node[lab, below left]{$c$} (R);
\draw[arr] (Qac) -- node[lab, left, pos=0.4]{$b$} (R);
\draw[arr] (Qbc) -- node[lab, below right]{$a$} (R);
\node[font=\footnotesize] at (2.6, 2.1) {$i+1$};
\node[font=\footnotesize] at (2.6, 1.4) {$i$};
\node[font=\footnotesize] at (2.6, 0.7) {$i-1$};
\node[font=\footnotesize] at (2.6, 0) {$i-2$};
\end{tikzpicture}
\]
The three lower diamonds have top vertices $P_a$, $P_b$ and $P_c$, of lower degree, so their relations hold by the induction hypothesis.
Writing $\tau_{ab}$, $\tau_{ac}$ and $\tau_{bc}$ for the signs of the arrows from $Q_{ab}$, $Q_{ac}$ and $Q_{bc}$ into $R$, these relations read
\[
\sigma_1\sigma_2=-\tau_{ab}\tau_{ac},\qquad
\sigma_3\sigma_4=-\tau_{ab}\tau_{bc},\qquad
\sigma_5\sigma_6=-\tau_{ac}\tau_{bc}.
\]
Multiplying them, each $\tau$ appears twice, so $\sigma_1\sigma_2\sigma_3\sigma_4\sigma_5\sigma_6=(-1)^3=-1$, which is condition \eqref{eq:sign-coherence}.
This completes the induction.
\end{proof}

\section{Schreyer frames and the Schreyer diagram}\label{app:schreyer}

In this appendix we explain how the Schreyer frames of La Scala and Stillman \cite{LaScala} led us to the valid words of Definition~\ref{def:valid_word}.
A Schreyer frame records the leading monomials of a free resolution computed with Gr\"obner bases, and over a quadratic monomial ring it has a simple combinatorial description.
In Section~\ref{ssec:frames} we recall the notions that we need and compute the frame of the running example.
In Section~\ref{ssec:frame-quadratic} we show that, with a suitable choice of orders, the frame of $R/(x_\Lambda)$ is the tree of the $(\mathsf G,\Lambda)$-valid words.
In Section~\ref{ssec:schreyer-words} we draw this tree directly as the Schreyer diagram, a variant of the Column-Diamond construction that completes each diamond with only one of its upper arrows, we obtain an independent proof of the formula for the Betti numbers in Proposition~\ref{prop:poincare_words}, and we show that the Schreyer diagram is the subdiagram of the Column-Diamond diagram formed by the last arrows of the paths with words $\omega(p)$.

\subsection{Schreyer frames}\label{ssec:frames}

We recall from \cite{LaScala} the notions that we need, keeping our notation $I$ for the ideal of $S$, which is denoted $J$ in \cite{LaScala}.
In this subsection $I\subseteq S$ is an arbitrary ideal and $R=S/I$.
We fix a term order $>$ on $S$, for instance the degree reverse lexicographic order, and we write ${\rm in}(I)$ for the monomial ideal generated by the leading monomials of the elements of $I$.
A monomial of $S$ is \emph{standard} if it does not lie in ${\rm in}(I)$.
When $I$ is a monomial ideal, as in the rest of the paper, ${\rm in}(I)=I$.

Let $F$ be a free $R$-module with a given basis, and let $\hat F$ be the free $S$-module with the same basis.
A \emph{monomial} of $\hat F$ is an element $t\cdot e$, where $t$ is a monomial of $S$ and $e$ is a basis element, and it is a \emph{monomial of $F$} when $t$ is standard.
A \emph{term order} on $F$ is a total order on the monomials of $\hat F$ such that $m<n$ implies $t\cdot m<t\cdot n$ for every monomial $t$ of $S$, and $s<t$ implies $s\cdot e<t\cdot e$ for every basis element $e$.
Let $M=F_0/L$, where $F_0$ is a free $R$-module and $L\subseteq F_0$ is a submodule, and consider a sequence of $R$-homomorphisms
\[
\Phi\colon\ \cdots\to F_l\xrightarrow{\varphi_l}F_{l-1}\xrightarrow{\varphi_{l-1}}\cdots\xrightarrow{\varphi_2}F_1\xrightarrow{\varphi_1}F_0 ,
\]
where each $F_i$ is a free $R$-module with a given basis $E_i$.
Note that $\Phi$ is not necessarily a chain complex.
Let $\tau=\{\tau_i\}$ be a sequence of term orders $\tau_i$ on $F_i$.
We say that $\tau$ is a \emph{term ordering on $\Phi$} if $s\cdot e_1<t\cdot e_2$ whenever $s\cdot{\rm lm}\,\varphi_i(e_1)<t\cdot{\rm lm}\,\varphi_i(e_2)$, for $e_1,e_2\in E_i$ and monomials $s,t$ of $S$ \cite[Definition~3.1]{LaScala}.

\begin{definition}[{\cite[Definition~3.2]{LaScala}}]\label{def:initial-terms}
Given $\Phi$ and a term ordering on $\Phi$, the \emph{initial terms} of $\Phi$ form the sequence
\[
{\rm in}(\Phi)\colon\ \cdots\to F_l\xrightarrow{\xi_l}F_{l-1}\xrightarrow{\xi_{l-1}}\cdots\xrightarrow{\xi_2}F_1\xrightarrow{\xi_1}F_0 ,
\]
where $\xi_i(e)={\rm lm}\,\varphi_i(e)$ for every $e\in E_i$.
That is, ${\rm in}(\Phi)$ replaces each column of each matrix of $\Phi$ by its leading monomial.
\end{definition}

\begin{definition}[{\cite[Definition~3.4]{LaScala}}]\label{def:schreyer-frame}
A \emph{Schreyer frame} of $M=F_0/L$ is a sequence of $R$-homomorphisms
\[
\Xi\colon\ \cdots\to F_l\xrightarrow{\xi_l}F_{l-1}\xrightarrow{\xi_{l-1}}\cdots\xrightarrow{\xi_2}F_1\xrightarrow{\xi_1}F_0 ,
\]
in which $\xi_i(e)$ is a monomial of $F_{i-1}$ for every $e\in E_i$, together with a term ordering on $\Xi$, such that
\begin{enumerate}
\item $\xi_1(E_1)$ is a minimal set of generators of ${\rm in}(L)$, and
\item $\xi_i(E_i)$ is a minimal set of generators of ${\rm in}(\ker\xi_{i-1})$ for $i\geq 2$.
\end{enumerate}
\end{definition}

Let $\Xi$ be a Schreyer frame.
For $e\in E_{i-1}$ we write
\[
E_i(e)=\{\epsilon\in E_i : \xi_i(\epsilon)=t\cdot e \text{ for some monomial } t \text{ of } S\}
\]
for the set of basis elements of $F_i$ that are sent to multiples of $e$, which may be empty.
Since every $\xi_i(\epsilon)$ is a monomial, each $\epsilon\in E_i$ lies in $E_i(e)$ for exactly one $e\in E_{i-1}$, so the sets $E_i(e)$, for $e\in E_{i-1}$, form a partition of $E_i$.

A Schreyer frame comes with a term ordering, that is, with a term order on every $F_i$, and these orders are far from independent.
By the compatibility condition, $s\cdot\epsilon$ and $t\cdot\epsilon'$ are compared through their images $s\cdot\xi_i(\epsilon)$ and $t\cdot\xi_i(\epsilon')$, and the comparison is left open only when these images coincide.
Since each $\xi_i(\epsilon)$ is a multiple of a single basis element of $F_{i-1}$, this can happen only when $\epsilon$ and $\epsilon'$ lie in the same set $E_i(e)$.
The next proposition shows that ordering these sets is the only choice involved.

\begin{proposition}[{\cite[Proposition~3.8]{LaScala}}]\label{prop:frame-orders}
Let $\Xi$ be a sequence of $R$-homomorphisms as in Definition~\ref{def:schreyer-frame}, in which $\xi_i(\epsilon)$ is a monomial of $F_{i-1}$ for every $\epsilon\in E_i$.
Giving a term ordering on $\Xi$ is equivalent to giving a term order on $F_0$ and a total order on each nonempty set $E_i(e)$.
Given these data, the term order on $F_i$ is defined inductively by $s\cdot\epsilon<t\cdot\epsilon'$ if either $s\cdot\xi_i(\epsilon)<t\cdot\xi_i(\epsilon')$ as monomials of $\hat F_{i-1}$, or $s\cdot\xi_i(\epsilon)=t\cdot\xi_i(\epsilon')$ and $\epsilon$ precedes $\epsilon'$ in the total order of $E_i(e)$, in which case $\epsilon$ and $\epsilon'$ lie in the same set $E_i(e)$.
\end{proposition}

Here a total order on $E_i(e)$ is an arbitrary ordering of this finite set, with no compatibility condition, and it is used only to break ties.

The next result is the key to computing a Schreyer frame one step at a time.
Let $\xi_i\colon F_i\to F_{i-1}$ be a homomorphism of free $R$-modules such that $\xi_i(\epsilon)$ is a monomial of $F_{i-1}$ for every $\epsilon\in E_i$, with term orders as in Proposition~\ref{prop:frame-orders}.
For $e\in E_{i-1}$ we write $E_i(e)=\{\epsilon_1,\dots,\epsilon_r\}$ numbered according to the total order of $E_i(e)$, and $\xi_i(\epsilon_j)=t_j\cdot e$.

\begin{lemma}[{\cite[Lemma~3.5]{LaScala}}]\label{lemma-in}
The module ${\rm in}(\ker\xi_i)$ is minimally generated by
\[
\bigcup_{e\in E_{i-1}}\ \bigcup_{j=1}^{r}\ {\rm mingens}\bigl(({\rm in}(I),t_1,\ldots,t_{j-1}):t_j\bigr)\cdot\epsilon_j ,
\]
where ${\rm mingens}$ denotes the set of minimal generators of the monomial ideal in parentheses that do not lie in ${\rm in}(I)$.\footnote{In \cite{LaScala} the inner union starts at $j=2$. When $I\neq 0$ it must start at $j=1$, since the colon ideal $({\rm in}(I):t_1)$ need not vanish.}
\end{lemma}

Finally, a frame determines the ranks of a free resolution.

\begin{proposition}[{\cite[Proposition~3.6]{LaScala}}]\label{prop:schreyer-resolution}
If $\Xi$ is a Schreyer frame of $M$, then there is a free resolution $\Phi$ of $M$, with a term ordering on $\Phi$, such that ${\rm in}(\Phi)=\Xi$.
In particular, the free modules of $\Phi$ have the same ranks as those of $\Xi$.
\end{proposition}

La Scala and Stillman call such a $\Phi$ a \emph{Schreyer resolution}, and observe that the proof is essentially Schreyer's original argument. Example~\ref{Schreyer-frame-example} below shows that different orders of the same set $E_i(e)$ lead to different frames.

\begin{example}\label{Schreyer-frame-example}
Let $S=k[x_1,x_2,x_3]$, $I=(x_1x_2,x_1x_3)$ and $R=S/I$, our running example, and let $M=R/(x_1)$, so that $F_0=R$ and $L=(x_1)$.
Since $I$ and $L$ are monomial, ${\rm in}(I)=I$ and ${\rm in}(L)=L$.
We write $e^i_j$ for the basis elements of $F_i$, where the superscript records the homological degree and the subscript the position from left to right in the diagram at the end of this example, and in each set $E_i(e)$ we order the basis elements by the index of the variable $t_j$, from the smallest to the largest.

The first map is immediate.
Since $\xi_1(E_1)$ must minimally generate $(x_1)$, we get $E_1=\{e^1_1\}$ and $\xi_1(e^1_1)=x_1e^0_1$.

We compute $\xi_2$ in full, since every later map repeats the same three moves of Lemma~\ref{lemma-in}.
First we run over $E_0=\{e^0_1\}$, so the outer union has a single term.
Then we collect the basis elements sent to multiples of $e^0_1$, which gives $E_1(e^0_1)=\{e^1_1\}$ with $t_1=x_1$, hence $r=1$.
Finally we compute the colon ideal $(I:x_1)=(x_2,x_3)$, whose minimal generators $x_2$ and $x_3$ do not lie in $I$.
Therefore ${\rm in}(\ker\xi_1)=(x_2e^1_1,\,x_3e^1_1)$, so $F_2=R^2$ and $\xi_2=\begin{bmatrix} x_2 & x_3\end{bmatrix}$.

The remaining maps follow the same three moves with more terms, so we only record their outcome.
For $\xi_3$ the outer union again has a single term, since $E_1=\{e^1_1\}$, and $E_2(e^1_1)=\{e^2_1,e^2_2\}$ with $t_1=x_2$ and $t_2=x_3$.
The colon ideals $(I:x_2)=(x_1)$ and $((I,x_2):x_3)=(x_1,x_2)$ give
\[
{\rm in}(\ker\xi_2)=(\,x_1e^2_1,\; x_2e^2_2,\; x_1e^2_2\,),
\]
so $F_3=R^3$ and
\[
\xi_3=\begin{bmatrix} x_1 & 0 & 0 \\ 0 & x_2 & x_1 \end{bmatrix}.
\]
For $\xi_4$ the outer union runs over $E_2=\{e^2_1,e^2_2\}$.
The set $E_3(e^2_1)=\{e^3_1\}$ contributes through $(I:x_1)=(x_2,x_3)$, and the set $E_3(e^2_2)=\{e^3_2,e^3_3\}$, ordered as $e^3_3,e^3_2$ with $t_1=x_1$ and $t_2=x_2$, contributes through $(I:x_1)=(x_2,x_3)$ and $((I,x_1):x_2)=(x_1)$, so that
\[
{\rm in}(\ker\xi_3)=(\,x_2e^3_1,\; x_3e^3_1,\; x_1e^3_2,\; x_2e^3_3,\; x_3e^3_3\,),
\]
so $F_4=R^5$ and
\[
\xi_4=\begin{bmatrix} x_2 & x_3 & 0 & 0 & 0 \\ 0 & 0 & x_1 & 0 & 0 \\ 0 & 0 & 0 & x_2 & x_3 \end{bmatrix}.
\]
Altogether the frame begins
\[
R^5 \xrightarrow{\begin{bmatrix} x_2 & x_3 & 0 & 0 & 0 \\ 0 & 0 & x_1 & 0 & 0 \\ 0 & 0 & 0 & x_2 & x_3 \end{bmatrix}} R^3 \xrightarrow{\begin{bmatrix} x_1 & 0 & 0 \\ 0 & x_2 & x_1 \end{bmatrix}} R^2 \xrightarrow{\begin{bmatrix} x_2 & x_3 \end{bmatrix}} R \xrightarrow{\begin{bmatrix} x_1\end{bmatrix}} R ,
\]
and every later map follows the same calculation.
Comparing with Example~\ref{ex:free-res-from-diag}, we see that $\xi_4=d_4$ and that $\xi_3$ is obtained from $d_3$ by deleting the entry $-x_3$, which is the dashed arrow of \eqref{diagrama-2-cubo}.
In diagram notation, with the colors of Example~\ref{ejemplo-2-cubo}, the frame is the following tree.
\[
\begin{tikzpicture}[xscale=-0.6, yscale=0.75, rotate=-90]
\node[dot] (n10) at (12, 0) {};
\node[dot] (n20) at (11, 0) {};
\node[dot] (n31) at (10, 2) {};
\node[dot] (n32) at (10,-2) {};
\node[dot] (n41) at (9, 4) {};
\node[dot] (n42) at (9, 0) {};
\node[dot] (n43) at (9,-4) {};
\node[dot] (n51) at (8, 5) {};
\node[dot] (n52) at (8, 3) {};
\node[dot] (n53) at (8, 0) {};
\node[dot] (n54) at (8,-3) {};
\node[dot] (n55) at (8,-5) {};

\draw[barr] (n20) -- (n10);
\draw[rarr] (n31) -- (n20);
\draw[garr] (n32) -- (n20);
\draw[barr] (n41) -- (n31);
\draw[rarr] (n42) -- (n32);
\draw[barr] (n43) -- (n32);
\draw[rarr] (n51) -- (n41);
\draw[garr] (n52) -- (n41);
\draw[barr] (n53) -- (n42);
\draw[rarr] (n54) -- (n43);
\draw[garr] (n55) -- (n43);
\end{tikzpicture}
\]

The tree depends on the orders chosen in the sets $E_i(e)$, as we see already with the syzygy $x_3e^2_1-x_2e^2_2$ of $\xi_2$.
Its two terms have the same image $x_2x_3e^1_1$, so its leading term is decided by the order of $E_2(e^1_1)$.
With our order the leading term is $x_2e^2_2$, and in the tree the arrow labeled $2$ lands on the arrow labeled $3$.
If we had ordered each set $E_i(e)$ from the largest index to the smallest, the leading term would be $x_3e^2_1$, the colon ideals at $E_2(e^1_1)$ would be $(I:x_3)=(x_1)$ and $((I,x_3):x_2)=(x_1,x_3)$, and the frame would be the following tree.
\[
\begin{tikzpicture}[xscale=-0.6, yscale=0.75, rotate=-90]
\node[dot] (n10) at (12, 0) {};
\node[dot] (n20) at (11, 0) {};
\node[dot] (n31) at (10, 2) {};
\node[dot] (n32) at (10,-2) {};
\node[dot] (n41) at (9, 4) {};
\node[dot] (n42) at (9, 0) {};
\node[dot] (n43) at (9,-4) {};
\node[dot] (n51) at (8, 5) {};
\node[dot] (n52) at (8, 3) {};
\node[dot] (n53) at (8, 0) {};
\node[dot] (n54) at (8,-3) {};
\node[dot] (n55) at (8,-5) {};

\draw[barr] (n20) -- (n10);
\draw[rarr] (n31) -- (n20);
\draw[garr] (n32) -- (n20);
\draw[barr] (n41) -- (n31);
\draw[garr] (n42) -- (n31);
\draw[barr] (n43) -- (n32);
\draw[rarr] (n51) -- (n41);
\draw[garr] (n52) -- (n41);
\draw[barr] (n53) -- (n42);
\draw[rarr] (n54) -- (n43);
\draw[garr] (n55) -- (n43);
\end{tikzpicture}
\]
Here the arrow labeled $3$ lands on the arrow labeled $2$.
The two trees have the same number of vertices in each homological degree, but they are different trees. \hfill$\diamond$

\end{example}

\subsection{Frames over quadratic monomial rings}\label{ssec:frame-quadratic}

From now on we return to the setting of the paper.
Let $R=S/I$ be a quadratic monomial ring with graph $\mathsf G$, let $\Lambda\subseteq[n]$ be nonempty with $N(\lambda)\neq\emptyset$ for every $\lambda\in\Lambda$, and let $M=R/(x_\Lambda)$, so that $F_0=R$ and $L=(x_\Lambda)$.
Since $I$ and $L$ are monomial ideals, ${\rm in}(I)=I$ and ${\rm in}(L)=L$.
The set $E_1$ has one element for each $\lambda\in\Lambda$, sent to $x_\lambda e^0_1$, so every entry of $\xi_1$ is a variable.

The following picture shows a basis element $e$ of $F_{i-1}$ and the set $E_i(e)=\{\epsilon_1,\ldots,\epsilon_r\}$ of basis elements of $F_i$ sent to multiples of $e$, with $\xi_i(\epsilon_l)=t_l\cdot e$.
\[
\begin{tikzpicture}[yscale=1.1]
\node[dot, black!30] (z1) at (-4, 6) {};
\node[dot, black!30] (z2) at (-3.5, 6) {};
\node[black!30, font=\footnotesize] at (-3, 6) {$\cdots$};
\node[dot, black!30] (z4) at (-2.5, 6) {};
\node[dot, black!30] (z10) at (-1.75,6) {};
\node[dot, black!30] (z20) at (-1.25,6) {};
\node[black!30, font=\footnotesize] at (-0.75,6) {$\cdots$};
\node[dot, black!30] (z40) at (-0.25,6) {};
\node[black!30] at (0.4, 6) {$\cdots$};
\node[dot] (z100) at (1, 6) {};
\node[font=\footnotesize] at (0.85,6.3) {$\epsilon_1$};
\node[font=\footnotesize] at (1.5, 6) {$\cdots$};
\node[dot] (z200) at (2, 6) {};
\node[font=\footnotesize] at (2, 6.3) {$\epsilon_j$};
\node[font=\footnotesize] at (2.75, 6) {$\cdots$};
\node[dot] (z500) at (3.5, 6) {};
\node[font=\footnotesize] at (3.65,6.3) {$\epsilon_r$};
\node[black!30] at (4, 6) {$\cdots$};
\node[dot, black!30] (zz1) at (4.5, 6) {};
\node[black!30, font=\footnotesize] at (5, 6) {$\cdots$};
\node[dot, black!30] (zz3) at (5.5, 6) {};
\node[dot, black!30] (zz4) at (6, 6) {};
\node[black!30, font=\footnotesize] at (6.5, 6) {$\cdots$};
\node[dot, black!30] (zz5) at (7, 6) {};
\node[dot, black!30] (a1) at (-3, 5) {};
\node[dot, black!30] (a2) at (-0.5,5) {};
\node[black!30] at (0.75,5) {$\cdots$};
\node[dot] (a3) at (2, 5) {};
\node[font=\footnotesize] at (2.25,4.75) {$e$};
\node[black!30] at (3, 5) {$\cdots$};
\node[dot, black!30] (a4) at (4, 5) {};
\node[dot, black!30] (b0) at (1, 4) {};
\node[black!30] at (1, 3) {$\vdots$};
\node at (-5.5,6) {$F_i$};
\node at (-5.5,5) {$F_{i-1}$};
\node at (-5.5,4) {$F_{i-2}$};
\draw[->] (-5.5,5.7) -- node[left]{\footnotesize $\xi_i$} (-5.5,5.3);
\draw[->] (-5.5,4.7) -- node[left]{\footnotesize $\xi_{i-1}$} (-5.5,4.3);
\draw[arr, black!30] (z1) -- (a1);
\draw[arr, black!30] (z2) -- (a1);
\draw[arr, black!30] (z4) -- (a1);
\draw[arr, black!30] (z10) -- (a2);
\draw[arr, black!30] (z20) -- (a2);
\draw[arr, black!30] (z40) -- (a2);
\draw[arr, black!30] (zz1) -- (a4);
\draw[arr, black!30] (zz3) -- (a4);
\draw[arr, black!30] (zz4) -- (a4);
\draw[arr, black!30] (zz5) -- (a4);
\draw[arr, black!30] (a1) -- (b0);
\draw[arr, black!30] (a2) -- (b0);
\draw[arr, black!30] (a4) -- (b0);
\draw[arr, black!30] (b0) -- (1,3.2);
\draw[arr] (z100) -- node[left, font=\footnotesize]{$t_1$} (a3);
\draw[arr] (z200) -- node[pos=0.35, right, font=\footnotesize]{$t_j$} (a3);
\draw[arr] (z500) -- node[right, font=\footnotesize]{$t_r$} (a3);
\draw[arr] (a3) -- (b0);
\end{tikzpicture}
\]

We first describe the colon ideals of Lemma~\ref{lemma-in} when the entries of $\xi_i$ are variables, as is the case for $\xi_1$.

\begin{lemma}\label{dia+col_lemma}
Let $e\in E_{i-1}$ and write $E_i(e)=\{\epsilon_1,\ldots,\epsilon_r\}$, numbered according to the total order of $E_i(e)$, with $\xi_i(\epsilon_l)=t_l\cdot e$, where $t_l=x_{s_l}$ is a variable.
Then, for $j=1,\ldots,r$,
\[
\bigl((I,t_1,\ldots,t_{j-1}):t_j\bigr)=I+(t_l : l<j,\ t_lt_j\notin I)+(x_s : x_sx_{s_j}\in I),
\]
and the minimal generators of this ideal that do not lie in $I$ are the variables $t_l$ with $l<j$ and $t_lt_j\notin I$, together with the variables $x_s$ with $x_sx_{s_j}\in I$.
\end{lemma}

\begin{proof}
The variables $t_1,\ldots,t_r$ are distinct, since $\xi_i(E_i)$ is a minimal set of generators.
Since $I$ is a monomial ideal and $t_1,\ldots,t_j$ are distinct variables, we have
\[
\bigl((I,t_1,\ldots,t_{j-1}):t_j\bigr)=(I:t_j)+(t_1,\ldots,t_{j-1}).
\]
Since $I$ is generated by quadratic monomials, $(I:t_j)$ is generated by $I$ together with the variables $x_s$ such that $x_sx_{s_j}\in I$.
A variable $t_l$ with $t_lt_j\in I$ already lies in $(I:t_j)$, so in the second summand it suffices to keep the $t_l$ with $t_lt_j\notin I$.
This gives the equality.
The minimal generators of the colon ideal are the variables listed, together with those minimal generators of $I$ not divisible by any of these variables. Since the latter lie in $I$, the minimal generators outside $I$ are exactly the variables listed.
For $j=1$ the second summand is empty.
\end{proof}

\begin{remark}\label{diagram-version-of-lemma-3.5}
In terms of diagrams, Lemma~\ref{dia+col_lemma} reads
\[
\bigl((I,t_{1},\ldots,t_{j-1}):t_j\bigr) = I+
\underbrace{(t_l : l<j,\ t_lt_j\notin I)}_{\text{diamonds}}+
\underbrace{(x_s : x_sx_{s_j}\in I)}_{\text{columns}} .
\]
The variables $x_s$ with $x_sx_{s_j}\in I$ give the columns on the arrow labeled $s_j$ into $e$.
Each variable $t_l$ with $l<j$ and $t_lt_j\notin I$ gives an arrow labeled $s_l$ landing on the arrow labeled $s_j$, which is one of the two upper arrows of the diamond on these two arrows.
The other upper arrow, labeled $s_j$ and landing on the arrow labeled $s_l$, does not occur in the frame.
\end{remark}

\begin{corollary}\label{cor:frame-variables}
Every nonzero entry of the Schreyer frame of $M$ is a variable.
In particular, $F_i\cong R(-i)^{|E_i|}$ as graded $R$-modules.
\end{corollary}

\begin{proof}
The entries of $\xi_1$ are variables.
If the nonzero entries of $\xi_i$ are variables, then by Lemma~\ref{lemma-in} the elements of $\xi_{i+1}(E_{i+1})$ are the products $u\cdot\epsilon_j$, where $u$ runs over the minimal generators not in $I$ of the colon ideals of Lemma~\ref{dia+col_lemma}, and these are variables.
By induction, every nonzero entry of the frame is a variable.
\end{proof}

By Corollary~\ref{cor:frame-variables}, each $\epsilon\in E_i$ with $i\geq 1$ satisfies $\xi_i(\epsilon)=x_s\cdot e$ for a unique $e\in E_{i-1}$ and a unique $s\in[n]$, and we call $s$ the \emph{label} of $\epsilon$.
The labels of the elements of each set $E_i(e)$ are distinct, by Lemma~\ref{dia+col_lemma}. We call the elements of a set $E_i(e)$ \emph{siblings}.
In diagram notation, siblings are the vertices of homological degree $i$ whose arrows land on the same vertex $e$, whether these arrows come from columns or from diamonds.
Following the labels down to $e^0_1$, we attach to every $\epsilon\in E_i$ a word $W_\epsilon$ of length $i$, defined by $W_{e^0_1}=\emptyset$ and $W_\epsilon=W_e\,s$ when $\epsilon\in E_i(e)$ has label $s$.
As in Definition~\ref{def:word-of-vertex}, words are read from homological degree $0$ upwards.

\begin{theorem}\label{thm:frame-valid-words}
Order each set $E_i(e)$ by the labels of its elements, from the smallest to the largest.
Then, for every $i\geq 0$ and every $\epsilon\in E_i$, the labels of the elements of $E_{i+1}(\epsilon)$ form the set $S_\Lambda(W_\epsilon)$.
Consequently, $\epsilon\mapsto W_\epsilon$ is a bijection between $E_i$ and the set of $(\mathsf G,\Lambda)$-valid words of length $i$.
\end{theorem}

\begin{proof}
We prove the first statement by induction on $i$.
For $i=0$, the set $\xi_1(E_1)$ minimally generates $(x_\Lambda)$, so the labels of the elements of $E_1=E_1(e^0_1)$ form the set $\Lambda=S_\Lambda(\emptyset)$.

Let $i\geq 1$ and $\epsilon\in E_i$, and let $e\in E_{i-1}$ be such that $\epsilon\in E_i(e)$.
By the induction hypothesis, the labels of the elements of $E_i(e)$ form the set $S_\Lambda(W_e)$.
Write $E_i(e)=\{\epsilon_1,\ldots,\epsilon_r\}$ with labels $s_1<\cdots<s_r$, and let $\epsilon=\epsilon_j$.
By Lemma~\ref{lemma-in}, the elements of $\xi_{i+1}(E_{i+1})$ that are multiples of $\epsilon_j$ are the products $u\cdot\epsilon_j$, where $u$ runs over the minimal generators not in $I$ of the colon ideal $((I,x_{s_1},\ldots,x_{s_{j-1}}):x_{s_j})$.
Since $x_sx_{s_j}\in I$ exactly when $s\in N(s_j)$, Lemma~\ref{dia+col_lemma} shows that the labels of the elements of $E_{i+1}(\epsilon_j)$ form the set
\[
N(s_j)\cup\{s_l : l<j,\ s_l\notin N(s_j)\}=N(s_j)\cup\{u\in S_\Lambda(W_e) : u<s_j\}=S_\Lambda(W_e\,s_j),
\]
where the first equality uses that the labels $s_l$ with $l<j$ are exactly the elements of $S_\Lambda(W_e)$ smaller than $s_j$.
Since $W_{\epsilon_j}=W_e\,s_j$, this proves the first statement.

For the second statement we argue by induction on $i$, the case $i=0$ being clear.
By Definition~\ref{def:valid_word}, the valid words of length $i\geq 1$ are the words $Ws$ with $W$ valid of length $i-1$ and $s\in S_\Lambda(W)$.
By induction, every valid word $W$ of length $i-1$ is $W_e$ for exactly one $e\in E_{i-1}$, and by the first statement the elements of $E_i(e)$ have distinct labels forming the set $S_\Lambda(W_e)$.
Hence every valid word of length $i$ is $W_\epsilon$ for exactly one $\epsilon\in E_i$.
\end{proof}
\begin{example}\label{ex:frame-core}
We illustrate Theorem~\ref{thm:frame-valid-words} with the running example, $M=R/(x_1)$ over $R=k[x_1,x_2,x_3]/(x_1x_2,x_1x_3)$, where $N(1)=\{2,3\}$ and $N(2)=N(3)=\{1\}$, ordering each set $E_i(e)$ as in Example~\ref{Schreyer-frame-example}.
The following table lists, for each $\epsilon$ of homological degree at most $3$, the two sets in the proof of Theorem~\ref{thm:frame-valid-words}, the columns $N(s_j)$ and the diamonds $\{s_l : l<j,\ s_l\notin N(s_j)\}$, whose union is $S_\Lambda(W_\epsilon)$.
The column of siblings lists the labels of the set $E_i(e)$ that contains $\epsilon$, that is, the labels of $\epsilon$ and of the other arrows landing on the same vertex $e$, while the last column lists the labels of the arrows landing on $\epsilon$.
\begin{center}
\renewcommand{\arraystretch}{1.2}
\begin{tabular}{ccccc}
$W_\epsilon$ & siblings & columns & diamonds & $S_\Lambda(W_\epsilon)$ \\
\hline
$1$ & $1$ & $\{2,3\}$ & $\emptyset$ & $\{2,3\}$ \\
$12$ & $2,3$ & $\{1\}$ & $\emptyset$ & $\{1\}$ \\
$13$ & $2,3$ & $\{1\}$ & $\{2\}$ & $\{1,2\}$ \\
$121$ & $1$ & $\{2,3\}$ & $\emptyset$ & $\{2,3\}$ \\
$131$ & $1,2$ & $\{2,3\}$ & $\emptyset$ & $\{2,3\}$ \\
$132$ & $1,2$ & $\{1\}$ & $\emptyset$ & $\{1\}$
\end{tabular}
\end{center}
The three behaviors of the proof appear in the table.
When $\epsilon$ is the first element of $E_i(e)$, as for $1$, $12$, $121$ and $131$, only columns occur.
For $13$ the smaller sibling $2$ is not a neighbour of $3$, and it gives a diamond.
For $132$ the smaller sibling $1$ is a neighbour of $2$, so the arrow labeled $1$ into $132$ is already a column and no diamond occurs.

In the tree of Example~\ref{Schreyer-frame-example}, the only arrow given by a diamond up to homological degree $4$ is the arrow labeled $2$ into $13$.
We draw dotted the other upper arrow of its diamond, labeled $3$ and landing on $12$, which does not occur in the frame.
\[
\begin{tikzpicture}[xscale=-0.6, yscale=0.75, rotate=-90]
\node[dot] (n10) at (12, 0) {};
\node[dot] (n20) at (11, 0) {};
\node[dot] (n31) at (10, 2) {};
\node[dot] (n32) at (10,-2) {};
\node[dot] (n41) at (9, 4) {};
\node[dot] (n42) at (9, 0) {};
\node[dot] (n43) at (9,-4) {};
\node[dot] (n51) at (8, 5) {};
\node[dot] (n52) at (8, 3) {};
\node[dot] (n53) at (8, 0) {};
\node[dot] (n54) at (8,-3) {};
\node[dot] (n55) at (8,-5) {};
\draw[barr] (n20) -- (n10);
\draw[rarr] (n31) -- (n20);
\draw[garr] (n32) -- (n20);
\draw[barr] (n41) -- (n31);
\draw[rarr] (n42) -- (n32);
\draw[barr] (n43) -- (n32);
\draw[rarr] (n51) -- (n41);
\draw[garr] (n52) -- (n41);
\draw[barr] (n53) -- (n42);
\draw[rarr] (n54) -- (n43);
\draw[garr] (n55) -- (n43);
\draw[garr, densely dotted, opacity=0.6] (n42) -- (n31);
\end{tikzpicture}
\]
Adding the dotted arrow completes the diamond of Example~\ref{ejemplo-2-cubo}, whose top vertex has the words $123$ and $132$.
The frame keeps only the path with word $132$.
\hfill$\diamond$
\end{example}

\subsection{The Schreyer diagram and valid words}\label{ssec:schreyer-words}

The frame of Theorem~\ref{thm:frame-valid-words} can also be drawn directly, by a small change in the Column-Diamond construction.

\begin{definition}\label{Schreyer-diagram}
Let $M=R/(x_\Lambda)$.
The \emph{Schreyer diagram} of $M$, denoted $S(R,M)$, is the labeled diagram obtained by running the Column-Diamond construction of $M$ with a single modification, namely that every diamond is completed with only one of its two upper arrows, the one carrying the smaller label.
All column completions and the ``No repetitions'' rule are applied as before, and no signs are attached.
\end{definition}

Every completion in the construction of $S(R,M)$ draws a single arrow, so every vertex of positive homological degree has exactly one outgoing arrow, and $S(R,M)$ is a tree.
For the running example, $S(R,M)$ is the first tree of Example~\ref{Schreyer-frame-example}.

\begin{proposition}\label{prop:schreyer-diagram-frame}
The Schreyer diagram $S(R,M)$ is the diagram of the frame of Theorem~\ref{thm:frame-valid-words}.
In particular, its vertices of homological degree $i$ are in bijection with the $(\mathsf G,\Lambda)$-valid words of length $i$.
\end{proposition}

\begin{proof}
Both are built one homological degree at a time, starting from the arrows labeled $\lambda\in\Lambda$ of the initial map.
Given the arrows into a vertex $e$, the construction of $S(R,M)$ attaches to each arrow labeled $t$ the columns labeled $s\in N(t)$ and, for each other arrow labeled $s<t$ into $e$ with $x_sx_t\notin I$, the upper arrow labeled $s$ of their diamond, which lands on the arrow labeled $t$.
By Remark~\ref{diagram-version-of-lemma-3.5}, these are exactly the arrows of the frame landing on the arrow labeled $t$ when $E_i(e)$ is ordered by labels.
The second statement is Theorem~\ref{thm:frame-valid-words}.
\end{proof}

\begin{corollary}\label{cor:betti-frame}
For every $i\geq 0$, the Betti number $\beta_i(M)$ is the number of $(\mathsf G,\Lambda)$-valid words of length $i$.
\end{corollary}

\begin{proof}
By Proposition~\ref{prop:schreyer-resolution} there is a free resolution $\Phi$ of $M$ with ${\rm in}(\Phi)=\Xi$, whose free modules have ranks $|E_i|$.
Since $R$ and $M$ are graded, we may choose $\Phi$ to be graded \cite[Section~3]{LaScala}, so each column of $\varphi_i$ is homogeneous, and its leading monomial is a variable times a basis element by Corollary~\ref{cor:frame-variables}.
Hence every entry of $\varphi_i$ is a linear form, $\Phi$ is minimal, and $\beta_i(M)=|E_i|$, which is the number of valid words of length $i$ by Theorem~\ref{thm:frame-valid-words}.
\end{proof}

This gives an alternative proof of the formula for the Betti numbers in Proposition~\ref{prop:poincare_words},
without using Lemma~\ref{lem:vertex-words},
and it is the route that first led us to the valid words.

We now locate the Schreyer diagram inside the Column-Diamond diagram, using the words $\omega(p)$ of Definition~\ref{def:word-of-vertex}.

\begin{proposition}\label{prop:schreyer-lex-max}
Let $p$ be a vertex of positive homological degree of the Column-Diamond diagram of $M$, and write $\omega(p)=Wv$.
Then the path from $p$ with word $\omega(p)$ starts with an arrow $p\xrightarrow{v}q$ such that $\omega(q)=W$, and it is the only arrow out of $p$ with this property.
The map $p\mapsto\omega(p)$ identifies the subdiagram formed by these arrows, one out of each vertex of positive homological degree, with the Schreyer diagram $S(R,M)$, preserving labels.
\end{proposition}

\begin{proof}
The path from $p$ with word $Wv$ starts with an arrow $p\xrightarrow{v}q$, and $W$ is a word of $q$.
Since $Wv$ is valid, so is $W$, and by Lemma~\ref{lem:vertex-words} we have $W=\omega(t)$ for some vertex $t$.
Since distinct vertices have no word in common, $t=q$, so $\omega(q)=W$.
If $p\xrightarrow{v'}q'$ is another arrow with $\omega(p)=\omega(q')\,v'$, then $v'=v$ and $\omega(q')=W$, so $q'=q$, and the arrow is the same by Corollary~\ref{cor:unique-arrow}.
By Lemma~\ref{lem:vertex-words}, the map $p\mapsto\omega(p)$ is a bijection onto the valid words, and the arrow just described goes from the vertex with word $Wv$ to the vertex with word $W$, with label $v$.
By Proposition~\ref{prop:schreyer-diagram-frame} and Theorem~\ref{thm:frame-valid-words}, these are exactly the vertices and arrows of $S(R,M)$.
\end{proof}

In the running example the arrows of the proposition are those of the first tree of Example~\ref{Schreyer-frame-example}, and the only arrow they omit up to homological degree $4$ is the dotted arrow of Example~\ref{ex:frame-core}.
In general, however, these arrows are not obtained from the Column-Diamond diagram by keeping, in each of its diamonds, the upper arrow with the smaller label. The rule of Definition \ref{Schreyer-diagram} applies to the diamonds as they arise in the construction of $S(R,M)$, which are fewer. 
Take the graph of Example~\ref{ex:valid-words}, with edges $\{1,2\}$, $\{1,3\}$, $\{2,4\}$, and $\Lambda=\{1,4\}$.
The vertex $p$ with $\omega(p)=1324$ also has the words $1234$ and $1243$, so its outgoing arrows are labeled $4$ and $3$, and they are the upper arrows of a diamond with bottom vertex $12$.
The arrow of the proposition is the one labeled $4$, the larger label.
Indeed, in the construction of $S(R,M)$ the arrow labeled $3$ into $12$ is never drawn, since the diamond on the arrows labeled $2$ and $3$ into $1$ keeps only the arrow labeled $2$.
Thus the diamond with top vertex $p$ does not occur in $S(R,M)$, and its arrow labeled $4$ arrives instead as a column on the arrow labeled $2$.

\bibliographystyle{alpha}

\end{document}